\documentclass[11pt,a4paper]{amsart}

\usepackage{amsfonts}
\usepackage{amsthm,amsmath,amssymb,mathrsfs,yhmath,dsfont}
\usepackage{pstricks,pst-node, pst-text,pst-3d,graphicx}
\usepackage{xcolor}
\usepackage{anysize}
\usepackage{enumerate}
\usepackage{verbatim}
\usepackage[active]{srcltx}
\usepackage{refstyle}
\usepackage{subfigure}
\usepackage{hyperref}

\marginsize{2.5cm}{2.5cm}{3cm}{3cm}
\newtheorem{theorem}{Theorem}[section]
\newtheorem*{thm}{Theorem}
\newtheorem{corollary}[theorem]{Corollary}
\newtheorem{lemma}[theorem]{Lemma}
\newtheorem{proposition}[theorem]{Proposition}
\newtheorem{example}[theorem]{Example}

\newtheorem{question}[theorem]{Question}

\theoremstyle{definition}

\newtheorem{definition}[theorem]{Definition}

\theoremstyle{remark}

\newtheorem*{rem}{Remark}

\numberwithin{equation}{section}

\newcommand{\hau}{\mathcal{H}}
\newcommand{\rr}{\mathbb{R}}
\newcommand{\nn}{\mathbb{N}}

\newcommand{\diam}{\textrm{diam}\: }

\newcommand{\e}{\mathbf{e}}
\newcommand{\ff}{\mathbf{f}}
\newcommand{\h}{\mathbf{h}}
\newcommand{\g}{\mathbf{g}}

\begin{document}

\title[Assouad dimension of self-similar sets with overlaps in $\rr^d$]{Assouad dimension and local structure of self-similar sets with overlaps in $\mathbf{\rr^d}$}
\author{Ignacio Garc\'{\i}a}
\address{CEMIM and IFIMAR\\
Universidad Nacional de Mar del Plata and CONICET, Mar del Plata, Argentina}
\email{nacholma@gmail.com}
\thanks{The author was partially supported by Universidad Nacional de Mar del Plata, grant EXA 902/18}

\maketitle
\begin{abstract}
For a self-similar set in $\rr^d$ that is the attractor of an iterated function system that does not verify the weak separation property, Fraser, Henderson, Olson and Robinson showed that its Assouad dimension is at least $1$. In this paper, it is shown that the Assouad dimension of such a set is the sum of the dimension of the vector space spanned by the set of \emph{overlapping directions} and the Assouad dimension of the orthogonal projection of the self-similar set onto the orthogonal complement of that vector space. This result is applied to give sufficient conditions on the orthogonal parts of the similarities so that the self-similar set has Assouad dimension bigger than $2$, and also to answer a question posed by Farkas and Fraser. The result is also extended to the context of graph directed self-similar sets. The proof of the result relies on finding an appropriate weak tangent to the set. This tangent is used to describe partially the topological structure of self-similar sets which are both attractors of an iterated function system not satisfying the weak separation property and of an iterated functions system satisfying the open set condition.
\end{abstract}

\section{Introduction}

Self-similar sets in $\rr^d$ are one the simplest constructions of sets involving fractal features. Let $\mathcal S=\{S_1,\ldots, S_m\}$ be a finite family of contractive similitudes on $\rr^d$, that is, $S_i(x)=c_iO_ix+b_i$, $1\le i\le m$, where $c_i\in(0,1)$, $O_i$ is a $d\times d$ orthogonal matrix and $b_i\in\rr^d$. The corresponding self-similar set is the unique non empty compact subset $F\subset\rr^d$ which verifies the identity
\begin{equation}\label{selfsimilar}
F=\bigcup_{i=1}^mS_i(F).
\end{equation}
The set of maps $\mathcal S$ is called iterated function system (IFS for short) and the set $F$ is also referred as the attractor associated to the system.

For a self-similar set, the classical notions of dimensions, such as Hausdorff and box-counting dimensions, coincide. Moreover, this common value can be determined or even approximated in case there is some condition that guarantees that the images of $F$ in (\ref{selfsimilar}) do not overlap {\em significantly}. For example, one of the most well-known separation properties is the open set condition, where, besides some other nice properties that hold under its assumption, the Hausdorff dimension of $F$ is the unique value that satisfies the Hutchinson-Moran formula. See \cite{Fal} and references therein for the basic definitions of the mentioned dimensions and results.

However, in absence of separation conditions, determine the Hausdorff dimension of a self-similar set is a challenging problem which in its full generality remains open. Important advances in this direction were obtained by Hochman in \cite{Ho, Ho2}. 

In this article, we consider the Assouad dimension of self-similar sets in $\rr^d$, $d>1$, which is motivated by the results of Fraser, Henderson, Olson and Robinson in \cite{FHOR}. The Assouad dimension is always an upper bound for the above mentioned dimensions, and was introduced by Assouad \cite{A1, A2} in connection with the problem of Lipschitz embeddings of metric spaces into Euclidean spaces. 
Given a non-empty bounded set $F\subset \mathbb{R}^d$, let $N_{r}(F)$ be the minimum number of closed balls of radius $r$ needed to cover $F$. 
By $B(x,R)$ we mean the closed ball of radius $R$ centered at $x$.
The \emph{Assouad dimension} of $F$ is defined as 
\begin{align*}
\dim_{A}F=\inf \ \Bigl\{\alpha :\ & \text{there is a constant }c\ 
\text{such that } \\
& \sup_{x\in F}N_{r}(B(x,R)\cap F)\leq c\left( \frac{R}{r}\right) ^{\alpha
}\ \text{for all }0<r<R \Bigr\}.
\end{align*}
Thus, $\dim_A F$ is an upper homogeneity exponent, in the sense that any ball in $F$ of radius $R$ can be covered roughly by $\epsilon^{-\dim_A F}$ balls of radius $\epsilon R$, for any $\epsilon>0$.

The behavior of the Assouad dimension of self-similar sets is strongly related to the weak separation property, which is another separation condition for iterated function systems that now we recall. 

Given an IFS $\mathcal S$ as above, define $\mathcal I=\{1,\ldots, m\}$ and let $\mathcal I^\ast=\bigcup_{k\ge1}\mathcal I^k$ be the set of finite words in the alphabet $\mathcal I$. Given $\alpha=(i_1,\ldots, i_k)\in\mathcal I^\ast$, we define $S_{\alpha}=S_{i_1}\circ\cdots\circ S_{i_k}$, and also for later reference let $O_\alpha=O_{i_1}\cdots O_{i_k}$ and $c_\alpha=c_{i_1}\cdots c_{i_k}$. Consider the subset $$\mathcal E=\{S_\alpha^{-1}\circ S_{\beta}: \alpha, \beta\in \mathcal I^{\ast}, \alpha\neq\beta\}$$ of the group of the similitudes on $\rr^d$, equipped with the topology of pointwise convergence. Let $F$ be the corresponding attractor and assume that it is not contained in any $(d-1)$-dimensional hyperplane. Then, the IFS $\{S_i\}_{i\in\mathcal I}$ satisfies the {\em weak separation property} ({WSP}) if the identity map $I$ is not an accumulation point of $\mathcal E$, that is, 
$$I\notin \overline{\mathcal E\setminus \{I\}}.$$

The WSP was first formulated by Lau and Ngai \cite{LN99} in a different (but equivalent) way, and was further studied by Zerner in \cite{Ze}, where many equivalent definitions were given. Recall that the stronger condition $I\notin \overline{\mathcal E}$ is equivalent to the well known open set condition, that will be defined later in the paper (see \cite{S94}).

The following result was proved by Fraser, Henderson, Olson and Robinson.

\begin{thm}\cite[Theorem 1.3]{FHOR}
Let $F\subset\rr^d$ be a self-similar set not contained in any $(d-1)$-dimensional hyperplane. If $F$ is the attractor of an IFS that satisfies the weak separation property, then $\dim_A F=\dim_H F$; otherwise, $\dim_A F\ge 1$.
\end{thm}
		
For self-similar sets in $\rr$, this result provides the precise dichotomy that its Assouad dimension is either its Hausdorff dimension, or maximal, depending on whether the weak separation property holds or not. Extensions in $\rr$ of this dichotomy were obtained in \cite{FO} for graph directed self-similar sets, and in \cite{AKT} for self-conformal sets.
 However, in $\rr^d$ the  above theorem is less precise and such a dichotomy is no longer true,  since there is a self-similar set not contained in any $(d-1)$-hyperplane that is the attractor of an IFS that does not satisfy the WSP but $\dim_H F<\dim_A F<2$; see \cite[Section 4]{FHOR} or Example \ref{Example} (a) in Section \ref{sectionexamples}. This shows that different phenomena may occur for the Assouad dimension of self-similar sets in $\rr^d$.

The aim of the present paper is to clarify the above situation. In particular, the main result  is an expression for the Assouad dimension of self-similar sets in $\rr^d$ not satisfying the weak separation property, which roughly speaking, is the sum of the dimension of the vector space spanned by the {\em overlapping directions} and the Assouad dimension of the projection of the set onto the orthogonal complement of that vector space.

\subsection{Overlapping directions, statement of the main result and some consequences.}

For a subset $E\subset\rr^d$, the notation $E\Subset\rr^d$ means that $E$ is not contained in any $(d-1)$-dimensional hyperplane. Also, recall that for a function $f:\rr^d\to\rr$, the argument of the minimum on a subset $A\subset\rr^d$ is the subset of $A$ defined by 
\begin{equation*}
\text{ argmin}\{f(x):x\in A\}:=\{y\in A: f(y)\le f(x) \text{ for all } x\in A\}.
\end{equation*} 

We denote by $\|\cdot\|$ the Euclidean norm in $\rr^d$ and by $\|\cdot\|_{\mathrm{op}}$ the corresponding operator norm on the space of $d\times d$ matrices. Also, let $\|\cdot\|_\infty=\|\cdot\|_{L^{\infty}([0,1]^d)}$. Given that the maps in $\mathcal E$ are similarities, then the IFS $\mathcal S$ does not satisfy the weak separation property if and only if there is a pair of sequences $\{\alpha_k\}$ and $\{\beta_k\}$ in $\mathcal I^\ast$ such that 
\begin{equation*}
0<\|S_{\alpha_k}^{-1}\circ S_{\beta_k}-I\|_{{\infty}}\to0 \ \ \textrm{as \ } k\to\infty.
\end{equation*}

\begin{definition}[Overlapping directions]
Let $\mathcal S=\{S_1, \ldots, S_m\}$ be an IFS with attractor $F\Subset\rr^d$. Assuming that $\mathcal S$ does not satisfy the weak separation property, let $\alpha_k, \beta_k\in\mathcal I^\ast$ be such that
$$0<\|\Phi_k\|_{{\infty}}\to0 \ \ \textrm{as } \ \ k\to\infty,$$ 
where $\Phi_k=S_{\alpha_k}^{-1}S_{\beta_k}-I$. Further assume, for some $a\in F$ and $\rho>0$, the following technical condition: 
\begin{equation}\label{boundderivative1}
\rho\|D\Phi_{k}\|_{\mathrm{op}}\le\|\Phi_{k}(a_{k})\| \ \ \textrm{for all } k,
\end{equation}
where ${a_k}$ is any point in ${\rm argmin}\{\|\Phi_k(x)\|: x\in B(a,\rho)\}$. An {\em overlapping direction} $\omega\in S^{d-1}$ for $\mathcal S$ is any limit point of the sequence $\{\Phi_k(a)/\|\Phi_k(a)\|\}_k$. 
We denote by $V_{\mathcal S}$ the vector subspace of $\rr^d$ spanned by the set overlapping directions for $\mathcal S$. 

Besides, if $\mathcal{S}$ verifies the weak separation condition, then there is no overlapping dimension and we define $V_{\mathcal S}=\{0\}$.
\end{definition}

\begin{rem}
The failure of the WSP for the IFS $\mathcal S$ assures that $\dim V_{\mathcal S}\ge1$; see Lemma \ref{lemmaFraser} below, which is a slight modification of \cite[Lemma 3.11]{FHOR}. Here, the hypothesis $F\Subset\rr^d$ is essential. For example, the ternary Cantor set can be embedded in $\rr^3$ as the attractor of an IFS with three maps: the two obvious similarities and an irrational rotation fixing the $x$-axis. This IFS does not satisfy the WSP, however it is easily seen that $V_{\mathcal S}=\{0\}$.
\end{rem}

In general, the definition of overlapping direction is not an intrinsic property of the attractor since it depends on the IFS. For example, the unit square $[0,1]^d$  may be seen as a self-similar set with an IFS satisfying the WSP, in which case there is no overlapping direction. But also, using ideas as in \cite[Section 2 (5)]{BG}, it is possible to construct for any $1\le k\le d$, an IFS $\mathcal S$ for $[0,1]^d$ such that $\dim V_{\mathcal S}=k$. For some other examples, which also contain line segments, see Section \ref{section:weaknoweak}.

\begin{question}
Suppose $F\Subset\rr^d$ is a totally disconnected self-similar set. 
Is the definition of overlapping direction intrinsic in this case? 
\end{question}

However, for a given self-similar set $F$, it is always possible to consider a generating IFS $\mathcal S$ with $\dim V_{\mathcal S}$ maximal. 
Moreover, if $\mathcal S_1$ and $\mathcal S_2$ are two such IFS with maximal dimension, then $V_{\mathcal S_1}=V_{\mathcal S_2}$ since the union of two IFS generating $F$ is again an IFS generating $F$. We denote by $V_{\mathcal O}$ this common subspace, which is intrinsic of the set.

\begin{definition}
The {\em overlapping dimension of $F$}, $\dim_{O}F$, is the dimension of the subspace $V_{\mathcal O}$, that is, $\dim_{O}F=\dim{V_{\mathcal O}}$. 	
\end{definition}

The following is the main result of the paper. Denote by $\pi_V:\rr^d\to V$ the orthogonal projection onto the subspace $V$. 

\begin{theorem}\label{paththeorem}
Let $F\Subset\rr^d$ be the attractor of the IFS $\mathcal S$. 
 For any subspace $V$ of $V_{\mathcal S}$ we have
\begin{equation}\label{formula}
\dim_A F=\dim V+\dim_{A} \pi_{V^{\perp}}(F). 
\end{equation}
In particular, 
\begin{equation}\label{formula1}
\dim_A F=\dim_O F+\dim_A \pi_{V_{\mathcal O}^{\perp}}(F). 
\end{equation}

Moreover, the strict inequality $\dim V<\dim_A F$ holds whenever $\dim V<d$.
\end{theorem}

Identity (\ref{formula1}), which is immediate from (\ref{formula}), express the Assouad dimension of $F$ in a way independent of the IFS. Note that (\ref{formula}) holds trivially in case that $\mathcal S$ satisfies the weak separation property.

In general, $\dim{V_\mathcal S}$ may not be easy to find, and the convenience of formula (\ref{formula}) is that it deals with any subspace of $V_{\mathcal S}$, which may simplify the calculation of the Assouad dimension. We provide some examples with homogeneous self-similar sets later in Section \ref{sectionexamples}. 

Theorem \ref{paththeorem} also holds in the more general context of graph directed self-similar sets, but for simplicity in the exposition we postpone its formulation and proof to Section \ref{section:graph}.

We do not know if the Assouad dimension of the projection appearing in (\ref{formula1}) is strictly greater than its Hausdorff dimension. See  Example \ref{Example} (c) in Section \ref{sectionexamples}.

\begin{question}\label{question:proy}
It is true that $\dim_A F=\dim_O F+\dim_H \pi_{V_{\mathcal O}^{\perp}}(F).$ 
\end{question}

In Section \ref{section:structure} we show that $V_{\mathcal S}$ is invariant under the orthogonal parts of the maps from $\mathcal S$. More precisely,  $Ov\in V_{\mathcal S}$ whenever $O\in\overline{\{O_\gamma:\gamma\in \mathcal I^\ast\}}$ and $v\in V_{\mathcal S}$; see Theorem \ref{teovec}, where it is also shown that any unit vector in $V_\mathcal S$ is an overlapping direction. In particular, the number of linearly independent vectors in the orbit $\{O^jv\}_{j\in\mathbb{Z}}$ gives a lower bound for the Assouad dimension of the set. Some particular situations are described in the next corollary.

We denote by $G_{d,k}$ the family of all orthogonal projections onto $k$-dimensional subspaces of $\rr^d$. 

\begin{corollary}\label{Coroconseq}
Assume that the self-similar set $F\Subset\rr^d$, $d\ge2$, is the attractor of an IFS that does not satisfy the weak separation property. 
\begin{enumerate}
\item If ${\{O_\gamma:\gamma\in \mathcal I^\ast\}}$ is dense in $O(d)$ or $SO(d)$, then $\dim_A F=d$. Moreover, for any $\pi\in G_{d,k}$ we have $\dim_A\pi F=k$. 
\item If there is an overlapping direction $v$ not contained in any invariant subspace of dimension $1$ of some  $O\in G$, then $\dim_A F\ge 2$. 
\item If $d=2$ and the orthogonal part of one of the similitudes from the IFS is a rotation by an angle different from $0$ and $\pi$, then  $\dim_A F=2$.
\item Assume $d=3$ and that the orthogonal part $O$ of one of the  similitudes from the IFS contains a rotation by an angle different from $0$ and $\pi$.
\begin{enumerate}
\item If there is an overlapping direction not contained in a non trivial invariant subspace of $O$, then $\dim_A F=3$. 
\item If an overlapping direction is contained in the $2$ dimensional invariant subspace of $O$, then $\dim_A F>2$.
\end{enumerate}
\end{enumerate}
\end{corollary}

The proofs of the above statements follow immediately from Theorem \ref{paththeorem} and Theorem \ref{teovec}, with the exception of the projection part from (1), which we postpone to the Remark at the end of Section \ref{section:tangent}.

\begin{rem}The second part of Corollary \ref{Coroconseq} (1) partially answers Question 2.11 from \cite{FO}, where it is asked if the same result holds but without additional (non) separation properties. See also \cite{FK} and \cite{O19} for some recent results related to the Assouad dimension of orthogonal projections.
\end{rem}

As another application of the main result, we answer a question posed by Farkas and Fraser in \cite{FF}. For a self-similar set $F\Subset\rr^d$, with $s=\dim_H F<1$, the following are equivalent (see \cite[Corollary 3.2]{FF}):
\begin{enumerate}[$(i)$]
\item {\em $F$ satisfies the weak separation property};
\item $\hau^s(F)>0$;
\item {\em the Hausdorff and Assouad dimensions of $F$ coincide.}
\end{enumerate}
Here, statement (i) means that there is an IFS for $F$ satisfying the WSP, which in case $s<1$ is equivalent to $(i')$: {\em any IFS for $F$ satisfies the WSP.} Also, the implications (i)$\Rightarrow$(ii)$\Rightarrow$(iii) hold without extra assumptions on the size of the Hausdorff dimension. 
However, given $d\in\nn\setminus\{1\}$ and $s\in(1, d]$, they also exhibit a self-similar set $F\Subset[0,1]^d$ with $\dim_H F=s$ such that one of its generating IFS fails the WSP but $\hau^s(F)>0$; see \cite[Proposition 3.3]{FF}. A natural question they pose to complete this study is (\cite[ Question 3.4]{FF}):

{\em Is it true that for all $d\in\mathbb N\setminus\{1\}$ there exists a self-similar set $F\Subset \rr^d$ such that $F$ fails the weak separation property, $\dim_H F=1$ and $\hau^1(F)>0$?}

We answer this question in the negative. 
 
\begin{corollary}
The answer to the above question is negative.  Moreover, also when $\dim_H F=1$ we have $(i)\Leftrightarrow(i')\Leftrightarrow(ii)\Leftrightarrow(iii)$.  
\end{corollary} 

\begin{proof}
If there exists such a self-similar set $F$, then the hypotheses $\dim_H F=1$ and $\hau^1(F)>0$ imply, by Corollary 3.1 in \cite{FF}, that $F$ is Ahlfors regular; in particular $\dim_A F=1$. But as $F$ is the attractor of an IFS $\mathcal S$ that does not satisfy the WSP, then $1\le\dim V_{\mathcal S}$. This implies $1<\dim_A F$ by Theorem \ref{paththeorem}, and therefore a contradiction. 
\end{proof}

We note that the self-similar set constructed in \cite[Proposition 3.3]{FF} is $F=[0,1]\times E^{d-1}$, where the unit interval is viewed as an homogeneous self-similar set with contraction $r$ and not satisfying the WSP, and $E^{d-1}$ is the Cartesian product of $d-1$ copies of the self-similar on $\rr$ with maps $x\to rx$ and $x\to rx+(1-r)$. Here $r\in(0,1/2]$ is chosen appropriately so that $\dim_H F=s$. Obviously, when $r=1/k$ for some natural $k>1$, there is an IFS for $[0,1]\times E^{d-1}$ that satisfies the open set condition.
\begin{question}
Given any $r\in(0,1/2]$, it is possible to find an IFS for $[0,1]\times E^{d-1}$ that verifies the weak separation property?
\end{question}
A negative answer when $d=2$ and $s<2$, implies that the definition of overlapping direction can be intrinsic even for self-similar sets that are not totally disconnected. Moreover, a negative answer implies that there is a self-similar set with $\dim_H F<2$, $\hau^{\dim_H F}(F)>0$ but no one of its generating IFS satisfies the WSP. Such a self-similar set with dimension $2$ exists by the example in \cite{CJPPS}. 

\subsection{Tangent structure.}\label{section:tangent} The proof of Theorem \ref{paththeorem} relies on the tangent structure of the self-similar set. More precisely, on the construction of an appropriate weak tangent of the set. Essentially, a weak tangent of a set is a limit point, in the Hausdorff metric, of a sequence of similar images of the set intersected with a fixed compact set. Weak tangents often have a simpler structure than that of the original set and can be used to estimate from below the Assouad dimension of the set. 

Recall that the Hausdorff distance between two non-empty compact sets $A, B\subset \rr^d$ is 
$$d_{\mathcal H}(A,B)=\inf\{\epsilon>0: A\subset [B]_\epsilon \ \textrm{and} \ \ B \subset [A]_\epsilon\},$$
where $[A]_\epsilon$ is the closed $\epsilon$-neighbourhood of $A$ defined by $$[A]_\epsilon=\{y\in\rr^d: \|a-y\|\le\epsilon\ \textrm{for some } a\in A\}.$$ 
This distance defines a complete metric on the space of non empty compact subsets of $\rr^d$. 

Let $E$ and $\hat{E}$ be non empty compact subsets of $\mathbb{R}^{d}$. We say that $\hat{E}$ is a \emph{weak tangent} of $E$ if there is a compact subset $X\subset \mathbb{R}^{d}$, that contains both $E$ and $\hat{E}$, and a
sequence of similarity maps $T_{k}:\mathbb{R}^{d}\rightarrow \mathbb{R}^{d}$ such that $T_{k}(E)\cap X\to_{d_{\mathcal H}}\hat{E}$ as $k\rightarrow \infty $. The usefulness of weak tangents is, as shown by Mackay and Tyson, that the following inequality holds 
\begin{equation}\label{boundtangent}
\dim _{A}\hat{E}\leq \dim _{A}E;
\end{equation}
see \cite[Prop. 2.9]{Ma} or \cite[Prop. 6.1.5]{MT}.

We assume that $F\Subset \rr^d$ is a self-similar set that is the attractor of an IFS $\mathcal S$ that does not verify	 the weak separation property. Consider $s$ linearly independent overlapping directions $\omega^1,\ldots, \omega^s$ for $\mathcal S$, generating a vector subspace that we call $V$, where $s\le\dim V_{\mathcal S}$. We denote by $W$ the $s$-parallelotope (that is, an $s$-dimensional generalization of a parallelogram) with the origin $\mathbf{0}$ as one of its vertex and whose edges, radiating from $\mathbf{0}$ and all of unit length, are  parallel to $\omega^1,\ldots, \omega^s$. 

Using techniques from \cite{FHOR}, it is possible to construct a weak tangent of the self-similar set which is convenient to deduce formula (\ref{formula}). Given $A,B\subset\rr^d$, recall the definition of the sumset $A+B=\{a+b: a\in A, b\in B\}$. Also, we write $W+x$ instead of $W+\{x\}$.  

\begin{theorem}\label{teotangent}
Under the above hypotheses and notation, there is a weak tangent of $F$ containing the set $W+F$.
\end{theorem} 

\begin{rem}
Now we can complete the proof the projection part of Corollary \ref{Coroconseq} $(1)$. The denseness hypothesis implies $\dim V_{\mathcal S}=d$, and then there is translation of the $d$-parallelotope $W$ contained in a weak tangent of $F$.   Because any projection of a weak tangent of a set is contained in a weak tangent of the projection of the set (see \cite[Section 3.3.1]{Fr18}), then, given any projection $\pi\in G_{d,k}$, we have that $\pi(W)$ is a subset of a weak tangent to $\pi(F)$, and consequently, $$k=\dim_A \pi(W)\le\dim_A \pi(F)\le k,$$ where we used (\ref{boundtangent}) in the second inequality.

Note that, alternatively, we can argue that there is a weak tangent of $F$ containing the unit cube $[0,1]^d$ since $\dim_A F=d$ (see \cite[Theorem 2.4]{FY3}). 
\end{rem}

\subsection*{Organization of the paper}  In Section \ref{section:prelims}, we give some preliminary definitions and results needed for the proof of Theorem \ref{teotangent}, which is given in Section \ref{section:proofteotangent}. Then, in Section \ref{section:structure}, we state some structural properties of the set of overlapping directions $V_{\mathcal S}$, and using this result, Theorem \ref{paththeorem} is deduced later in that section. In Section \ref{sectionexamples} we give some examples and in Section \ref{section:graph} we extend Theorem \ref{paththeorem} to the context of graph directed self-similar sets. 

A final application of the result on the tangent structure is given in Section \ref{section:weaknoweak}. There, we consider the topological structure of a class of self-similar sets which are both overlapping and non overlapping. More precisely, we consider the self-similar sets that are attractors of iterated function systems not satisfying the WSP and also, that are attractors of iterated function systems satisfying the open set condition. We show in Theorem \ref{Theorem:product-structure} that a self-similar set in this class contains the Cartesian product of a non trivial cube in the maximal overlapping vector space $V_{\mathcal O}$  with a set whose Assouad dimension is $\dim_A\pi_{V_{\mathcal O}^\perp}(F)$.

\section{Some technical preliminaries}\label{section:prelims}
We begin this section showing that the set of overlapping directions is non empty if the weak separation condition does not hold. The following lemma  is a slight modification of \cite[Lemma 3.11]{FHOR}, which we refer for the proof. Recall that $\|\cdot\|_\infty=\|\cdot\|_{L^\infty([0,1]^d)}$. 

\begin{lemma}\label{lemmaFraser}
Let $\Phi_k$ be a sequence of affine maps in $\rr^d$ such that $0<\|\Phi_k\|_{{\infty}}\to0$ as $k\to \infty$. Then, given any set $\Gamma\Subset\rr^d$, there exist $\rho>0$, $a\in \Gamma$ and a subsequence $k_j$ such that for every $j$ we have
\begin{equation}\label{boundderivative}
\rho\|D\Phi_{k_j}\|_{\mathrm{op}}\le\|\Phi_{k_j}(a_{k_j})\|, 
\end{equation}
where $a_k$ is any point in ${\rm argmin}\{\|\Phi_k(x)\|: x\in B(a,\rho)\}$.
\end{lemma}

Inequality (\ref{boundderivative}) is the technical condition (\ref{boundderivative1}) from the definition of overlapping direction when $\Phi_k=S_{\alpha_k}^{-1}\circ S_{\beta_k}-I$. 

\begin{corollary}
If $\mathcal S$ is an IFS not satisfying the weak separation property and with attractor $F\Subset\rr^d$, then the set of overlapping directions is non empty. 
\end{corollary}

Now we turn to the preliminaries for the proof of Theorem \ref{teotangent}. 
We will need the following two lemmas.

\begin{lemma}\label{lemma:ultimo}
In the definition of overlapping direction, the technical condition (\ref{boundderivative1}) implies for each $k$ that 
\begin{equation*}
\max\bigl\{\|\Phi_{k}(x)\|: x\in B(a,\rho)\bigr\}\le 3\min\bigl\{\|\Phi_{k}(x)\|: x\in B(a,\rho)\bigr\}.
\end{equation*}  
\end{lemma}

\begin{proof}
This is shown at the end of the proof of \cite[Lemma 3.11]{FHOR}.  Pick $y\in B(a,\rho)$ realizing the minimum and let $x\in B(a,\rho)$. Then 
\begin{equation*}
\|\Phi_k(x)-\Phi_k(y)\|\le\|D\Phi_k\|_{\mathrm{op}}\:\|x-y\|\le 2\min\bigl\{\|\Phi_{k}(x)\|: x\in B(a,\rho)\bigr\}
\end{equation*} 
and the Lemma follows easily.
\end{proof}

\begin{lemma}\label{lemmaopposite} If $\omega$ is an overlapping direction for $\mathcal S$, then so is $-\omega$.
\end{lemma}

\begin{proof} 
Consider a pair of sequences $\{\alpha_k\}$ and $\{\beta_k\}$ in $\mathcal I^\ast$ such that 
\begin{equation*}
0<\|\Phi_{\alpha_k, \beta_k}\|_{{\infty}}\to0 \ \ \textrm{as \ } k\to\infty,
\end{equation*}
where $\Phi_{\alpha_k, \beta_k}=S_{\alpha_k}^{-1}\circ S_{\beta_k}-I$.
Then, it is easily seen that $c_{\beta_k}c_{\alpha_k}^{-1}\to1$ and $O_{\alpha_k}^{-1}O_{\beta_k}\to I$, since the convergence is uniform on bounded sets.
Also note that for any $x\in\rr^d$,
\begin{align*}\Phi_{\beta_k, \alpha_k}(x)&= c_{\beta_k}^{-1}O_{\beta_k}^{-1}(S_{\alpha_k}x-S_{\beta_k}x)\\ 
&  =-\frac{c_{\alpha_k}}{c_{\beta_k}}O_{\beta_k}^{-1}O_{\alpha_k}c_{\alpha_k}^{-1}O_{\alpha_k}^{-1}(S_{\beta_k}x-S_{\alpha_k}x) \\ &= -\frac{c_{\alpha_k}}{c_{\beta_k}}O_{\beta_k}^{-1}O_{\alpha_k}(\Phi_{\alpha_k, \beta_k}(x)).
\end{align*}
 Then $\|\Phi_{\alpha_k, \beta_k}\|_{{\infty}}/\|\Phi_{\beta_k, \alpha_k}\|_{{\infty}}\to1$, and consequently $0<\|\Phi_{\beta_k, \alpha_k}\|_{{\infty}}\to0$. Moreover, if $$\Phi_{\alpha_k, \beta_k}(a)/\|\Phi_{\alpha_k, \beta_k}(a)\|\to\omega$$ and condition (\ref{boundderivative1}) holds for some $\rho>0$, then $\Phi_{\beta_k, \alpha_k}(a)/\|\Phi_{\beta_k, \alpha_k}(a)\|\to-\omega$, and (\ref{boundderivative1}) holds for some slightly smaller value than $\rho$. 
\end{proof}

\subsubsection*{{\bf Preliminary definitions and results from \cite{FHOR}.}}

The proof of Theorem \ref{teotangent} uses techniques introduced for the proof  of Theorem 3.2 in \cite{FHOR}, so in the rest of this section we borrow notation and refer to some results given there. Also, for simplicity, we assume here $s=2$. 

We will define, inductively, maps $g_j$ and $h_j$ that are essential for the definition of the weak tangent.
 
For notational convenience, let $\omega^0$  and $\omega^2$ be the two chosen overlapping directions. For $i=0,2$, let $\alpha_k^{(i)}, \beta_k^{(i)}\in\mathcal I^\ast$ be the words corresponding to the definition of $\omega^i$ and, denoting by $\Phi_{i, k}=\Phi_{\alpha_k^{(i)},\beta_k^{(i)}}$ the corresponding affine map, inequality  (\ref{boundderivative1}) holds
with $a^i\in F$. Obviously, the same $\rho>0$ can be chosen for both sequences. 
Also for ease of notation, we assume that $a^i$ is a fixed point of the similarities, that is,  there is $\gamma^{(i)}\in\mathcal I^\ast$ such that $S_\gamma^{(i)}(a^i)=a^i$. See the Remark at the end of this section in case $a^i$ is not a fixed point. 

Define
\begin{equation*}
\omega^1:=-\omega^0, \ a^1:=a^0, \ \alpha^{(1)}_k:=\beta^{(0)}_k, \ \beta^{(1)}_k:=\alpha^{(0)}_k, \  \gamma^{(1)}:=\gamma^{(0)}, \  \Phi_{1,k}:=\Phi_{\beta_k, \alpha_k},
\end{equation*}
and, for $0\le i\le2$, let 
\begin{align*}
\delta_k^{(i)}&=\min\bigl\{\|\Phi_{i, k}(x)\| : x\in B(a^{i},\rho)\bigr\} \notag 
\end{align*}
so that, by Lemma \ref{lemma:ultimo} and Lemma \ref{lemmaopposite},
\begin{equation}\label{deltaineq}
\max\bigl\{\|\Phi_{i, k}(x)\| : x\in B(a^{i},\rho)\bigr\}\le3\delta_k^{(i)}.
\end{equation}

Fix $n$ large and let $\epsilon=1/n^{s+2}$ and $\eta=\epsilon^2$.
By taking further subsequences, we assume that $\Phi_k(a^i)/\|\Phi_k(a^i)\|\in B(\omega^i,{\eta/2})$ for all $k$ and $i$. We consider values of $n$ large enough so that $B(\omega^i, {\eta/2})\cap B(\omega^{j}, {\eta/2})=\emptyset$ whenever $i\neq j$.

Define $f_0=f_1=S_{\gamma^{(0)}}$ and $f_2=S_{\gamma^{(2)}}$ and pick $M>0$ so that 
\begin{equation}\label{M}
f_i^m(x)\in B(a^i,{\rho'}) \ \ \ \textrm{ for any }x\in F \textrm{ and } m\ge M,
\end{equation}
where $\rho'=\rho\eta/5$.

Denote by $O(d)$ the group of orthogonal $d\times d$ matrices with entries in $\rr$ and let $G=\overline{\{O_\gamma: \gamma\in\mathcal I^\ast\}}$, which is a compact subgroup of $O(d)$; see Lemma 3.9 in \cite{FHOR}. Then, for $\epsilon_2=\epsilon^2/4$, there is a finite subset $\mathcal J\subset\mathcal I^\ast$ such that for any $U\in G$ there is an $\alpha\in\mathcal J$ such that $\|U-O_\alpha\|_{\mathrm{op}}<\epsilon_2$. Define 
\begin{equation*}
c_\ast=\min\{c_\alpha: \alpha\in \mathcal J\}.
\end{equation*} 
Note that $c_\ast$ depends on $n$ if the group $G$ is not finite.        

Now we define the maps $g_j$ and $h_j$. Let $g_0=h_0=I$, $d_0=1$, $A_0=I$. Inductively we define natural numbers $k_j$, $m_j$ and maps $g_j$ and $h_j$. For these definitions we need to introduce an additional sequence of integers  $\{p_l\}$, that will be specified later, and is strictly increasing with $p_0=0$.  Assume we have defined $g_{j-1}$ and $h_{j-1}$, and let $d_{j-1}$ be the ratio of $g_{j-1}$ and $A_{j-1}$ be the orthogonal matrix given by $Dg_{j-1}=d_{j-1}A_{j-1}$. Suppose $l$ is such that $j\in (p_l, p_{l+1}]$, and let $i\in\{0,1,2\}$ be such that $l\equiv i\mod 3$. Then, we choose $k_j$ and $m_j\ge M$ so that
\begin{equation}\label{fitpaso}
d_{j-1}c_\ast^{-1}c_{\gamma^{(i)}}^{-m_j}\delta_{k_j}^{(i)}<\epsilon\le d_{j-1}c_\ast^{-1}c_{\gamma^{(i)}}^{-m_j-1}\delta_{k_j}^{(i)};
\end{equation}
one can first choose $k_j$ large enough so that $d_{j-1}c_\ast^{-1}\delta_{k_j}^{(i)}<\epsilon c_{\gamma^{(i)}}^{M}$ and then choose $m_j$ appropriately. Also, 
pick $\nu_j\in\mathcal J$ such that 
\begin{equation*}
\|A_{j-1}O_{\nu_j}^{-1}O_{\gamma^{(i)}}^{-m_j}-I\|_{\mathrm{op}}<\epsilon_2.
\end{equation*}
Then, define the functions 
\begin{equation}\label{gjhj}
g_j=g_{j-1}\circ S_{\nu_j}^{-1}\circ f_i^{-m_j}\circ S_{\alpha_{k_j}^{(i)}}^{-1} \ \ \ \textrm{and} \ \ \ h_j=S_{\beta_{k_j}^{(i)}}\circ f_i^{m_j}\circ S_{\nu_j}\circ h_{j-1}.
\end{equation}

As shown in \cite{FHOR}, for any $j\ge1$ and $x\in F$ we have
\begin{equation}\label{esta}
g_j\circ h_j(x)-g_{j-1}\circ h_{j-1}(x)=d_{j-1}c_{\nu_j}^{-1}c_{\gamma^{(i)}}^{-m_j}A_{j-1}O_{\nu_j}^{-1}O_{\gamma^{(i)}}^{-m_j}	\Phi_{i,k_j}\bigl(f_i^{m_j}\circ S_{\nu_j}\circ h_{j-1}(x)\bigr).
\end{equation}
Identity (\ref{esta}) has two important consequences. The first is, by (\ref{deltaineq}) and  (\ref{fitpaso}), that
\begin{equation}\label{increase}
c_\ast c_{\gamma^{(i)}}\epsilon\le \bigl\|g_j\circ h_j(x)-g_{j-1}\circ h_{j-1}(x)\bigr\|\le 3\epsilon,
\end{equation}
while the second is, using (\ref{boundderivative}),  that if $j\in(p_l,p_{l+1}]$ and $l\equiv i\mod 3$, then
\begin{equation}\label{cone}
g_j\circ h_j(x)-g_{j-1}\circ h_{j-1}(x)\in\mathcal C\bigl( B(\omega^i, \eta)\bigr),
\end{equation}	
where, given $U\subset \rr^d$, the set $\mathcal C(U)=\{\lambda u: \lambda>0, u\in U\}$ is the cone generated by $U$. For $y\in\rr^d$, let $\mathcal C_y(U)=\mathcal{C}(U)+y$, that is the translation of the cone generated by $U$ with vertex $y$. With this notation, and letting $y_j=g_j\circ h_j(x)$, then (\ref{cone}) says that $y_j\in \mathcal C_{y_{j-1}}(B(\omega^i,\eta))$.

We will also use the fact that, for a ball $U$ not containing the origin, then 
\begin{equation}\label{propertycone}
x\in\mathcal C_y(U), \ \textrm{implies} \ \mathcal C_x(U)\subset\mathcal C_y(U).
\end{equation}

\begin{rem} The constructions given below rely on (\ref{increase}) and (\ref{cone}), and are made, for notational convenience, assuming that the points $a^i$ are fixed points of compositions of the similarities from the IFS. In general, if $a^i$ is not a fixed point,  
there is an infinite word $\gamma$ such that $\lim_{n\to\infty} S_{\gamma|_n}(0)=a^i$,  where  $\gamma|_n\in\mathcal I^n$ is the word given by the first $n$ terms of $\gamma$. We define $M$ as in (\ref{M}) but requiring that  $S_{\gamma|_m}(F)\subset B(a^i,{\rho'})$ for any $m\ge M$. Inequalities (\ref{fitpaso}) and the definition of the functions in (\ref{gjhj}) are changed accordingly. Then, only the lower bound in (\ref{increase}) is modified, becoming $c_\ast c_{\min}\epsilon$; here $c_{\min}$ is the minimum of the contractions ratios form the IFS. Also, a quick inspection in \cite[proof of Theorem 3.2, page 207]{FHOR} shows that  property (\ref{cone}) still holds. Although these changes may modify slightly the sets constructed below, the fundamental property (\ref{b}) still holds.
\end{rem}

\section{Proof of Theorem \ref{teotangent}.}\label{section:proofteotangent}

\subsection*{Sketch of proof of Theorem \ref{teotangent}}

The summary of the proof is as follows. Let $\omega^1,\ldots, \omega^{s}$ be $s$ linearly independent overlapping directions for $\mathcal S$ that generate the $s$-parallelotope $W$; here $1\le s\le\dim V_{\mathcal S}$. Given $n\in\nn$ and $x_0\in F$, we construct a set $W_n=W_n(x_0)$ which is image of points in $F$ through a similarity $T_n$, and such that $W_n\to_{d_{\mathcal H}} W+x_0$ as $n\to \infty$ and $W_n\subset X:=[W]_1$. Next we show that essentially the same construction gives, for each $x\in F$, a set $W_n(x)\subset T_k(F)\cap X$ such that 
\begin{equation}\label{b}
W_n(x)\to_{d_{\mathcal H}} W+x \text{ \ as \ }n\to \infty.
\end{equation} 
Then, the theorem follows easily from the compactness of the space of non-empty compact subsets of $X$ with respect to the Hausdorff metric.
\vspace{.2cm}

Now we start with the construction of the sets.

\subsubsection*{{\bf Construction of $\mathbf{W_n}$.}}  We make first the construction for the case of two independent overlapping directions; $s=2$.
Then we indicate how to proceed in the general case. The set is constructed inductively using points of the form $g_j\circ h_j(x_0)$ for a fixed $x_0\in F$. The way  we choose to do it, although far from optimal in the sense of the number of required iterations, seems appropriate for the general case $s>2$.  

The construction is made in $n+1$ steps defining a `comb-like path'; see Figure \ref{fig:peine}. Each step specifies a `comb tooth' and three terms of a (finite) sequence $\{p_l\}$.
Below we detail the first step, see Figure \ref{fig:peinepaso1}. This step already gives the construction of $W_n$ when $s=1$, in which case the parallelotope $W$ is a unit line segment.
  
\noindent {\bf Step 1.} Fix $x_0\in F$ and let $a_1=x_0$; this is the base point of the first comb tooth.
First we use (\ref{cone}) with $j=1$:
as $p_0=0$, we have $p_0<1\le p_1$ ($p_1$ to be specified below), so $i=0$. Then, the point $g_1\circ h_1(x_0)$ belongs to the cone $\mathcal{C}_{a_1}\bigr(B(\omega^0,\eta)\bigl)$, and also by (\ref{increase}), its orthogonal projection onto the axis of the cone, which has direction $\omega^0$, is at least $c_\ast c_{\gamma^{(0)}} \sqrt{1-\eta^2}\epsilon$ and at most $3\epsilon$ apart from $a_1$. 

Applying inductively (\ref{cone}) and  (\ref{propertycone}), we have 
$$g_j\circ h_j(x_0)\in \mathcal C_{a_1}\bigr(B(\omega^0,\eta)\bigl)$$ 
for any $1\le j\le p_1$, and the orthogonal projection onto the axis  increases with $j$ in the direction $\omega^0$, with bounds given by  
\begin{equation}\label{spread}
c_\ast c_{\gamma^{(0)}}\sqrt{1-\eta^2}\epsilon \le \omega^0\cdot(g_j\circ h_j(x_0)-g_{j-1}\circ h_{j-1}(x_0)) \le3\epsilon.
\end{equation}
Then, since the lower bound of this increase is independent of $j$, we define $p_1$ as the largest integer $j$ such that  
\begin{equation}\label{stretch1}
g_j\circ h_j(x_0)\in \mathcal C_{a_1}\bigr(B(\omega^0,\eta)\bigl)\cap B(a_1,1),
\end{equation} 
and also define $$\tilde W_{n,1}=\{g_j\circ h_j(x_0):  p_0\le j\le p_1\},$$
which corresponds to one comb tooth.

Let $L_1$ and $L_2$ be the sides of the parallelogram $W$ containing $\mathbf{0}$ and parallel to $\omega^0$ and $\omega^2$, respectively. By (\ref{stretch1}) and the upper bound in (\ref{spread}), together with the fact that the width of the cone is $\eta$, we get
\begin{equation*}
d_{\mathcal H}\big(\tilde W_{n,1}, L_1+a_1\big)\le3\epsilon+\eta.
\end{equation*}

Next we move back to a point $\tilde a_1$ close to $a_1$ proceeding as follows. Put $b=g_{p_1}\circ h_{p_1}(x_0)$. We use (\ref{cone}) for  $j\in(p_1, p_2]$ ($p_2$ defined below), so $i=1$ and the involved direction is $\omega^1=-\omega^0$: the point $g_j\circ h_j(x_0)$ lie in the cone $\mathcal C_{b}\bigr(B(\omega^1,\eta)\bigl)$ and its projection onto the axis of this cone increases with $j$ in the direction $\omega^2$, with the same bounds as in (\ref{spread}). Define $p_2$ as the largest $j> p_1$ satisfying the same condition as in (\ref{stretch1}) but replacing $a_1$ and $\omega^0$ by $b$ and $\omega^1$, respectively. 
Put $\tilde a_1=g_{p_2}\circ h_{p_2}(x_0)$ and note that, since we are coming back in the direction $\omega^1$, we have by (\ref{increase}) that 
\begin{equation}\label{equ}
\|\tilde a_1 -a_1\|\le 3\epsilon+2\eta.
\end{equation}

Then, we continue moving slightly in the direction $\omega^2$ to construct the base point, $a_2$, of the second comb tooth: now $j\in (p_2,  p_3]$ ($p_3$ defined below), so $i=2$ and the corresponding points lie inside $\mathcal C_{\tilde a_1}\bigr(B(\omega^2,\eta)\bigl)$. We define $p_3$ by picking the largest $j>p_2$ satisfying
\begin{equation*}\label{stretch2}
g_j\circ h_j(x_0)\in \mathcal C_{\tilde a_1}\bigr(B(\omega^2,\eta)\bigl)\cap B\bigl(\tilde a_1, 1/n\bigr).
\end{equation*}
We conclude the first step defining $a_2:=g_{p_3}\circ h_{p_3}(x_0)$. 
Note that here we moved from $\tilde a_1$ about $1/n$ in the direction $\omega^2$  so that the distance between the comb teeth is about $1/n$.
    
\begin{figure}\label{peine}
    \centering
    \subfigure[The comb-like shape of $W_n$.]
    {
        \includegraphics[width=3in]{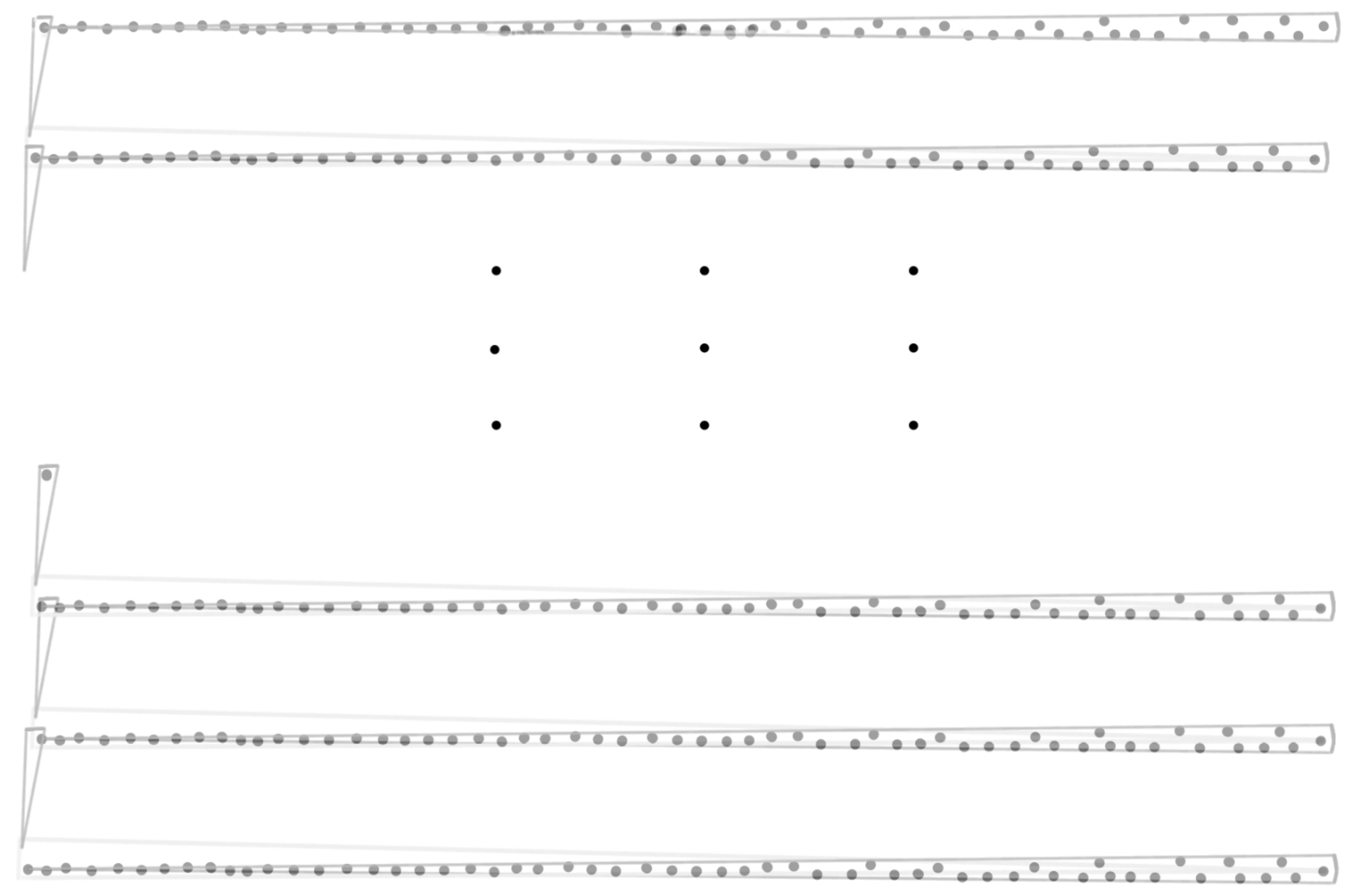}
        \label{fig:peine}
    }   
\\    
    \subfigure[The points defined at Step 1. Here $\mathcal C_{a_1}=\mathcal C_{a_1}\bigl(B(\omega^0,\eta)\bigr)\cap B(a_1,1)$, $\mathcal C_{b}=\mathcal C_{b}\bigl(B(\omega^1,\eta)\bigr)\cap B(b,1)$ and $\mathcal C_{\tilde a_1}=\mathcal C_{\tilde a_1}\bigl(B(\omega^2,\eta)\bigr)\cap B(\tilde a_1,1/n)$.]
    {
    \includegraphics[width=4in]{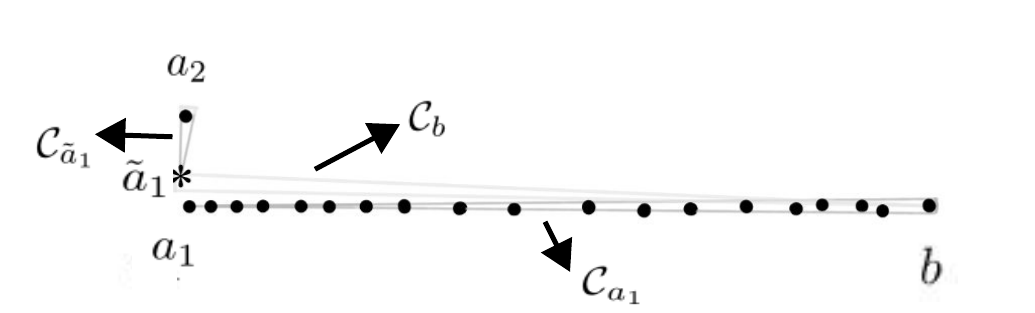}
        \label{fig:peinepaso1}
    }
  \caption{The set $W_n$.} 
\end{figure}

\noindent{\bf Remaining steps.}  At step $K$, the path begins at the point $a_{K}=g_{p_{3(K-1)}}\circ h_{p_{3(K-1)}}(x_0)$ defined at the end of the previous step.  
For $1< K \le n$, we follow the same pattern as before to define $p_{3K-2}, p_{3K-1}, p_{3K}$ and the corresponding set 
$$\tilde W_{n,K}=\{g_j\circ h_j(x_0):p_{3(K-1)}\le j\le p_{3K-2}\},$$
so that 
\begin{equation}\label{distan}
d_{\mathcal H}\big(\tilde W_{n,K}, L_1+a_K\big)\le 3\epsilon+\eta.
\end{equation}

In the last step, $n+1$, the construction ends after defining $\tilde W_{n,n+1}$, so we do not go back to a point close to $a_{n+1}$. 

We define $$W_n=\bigcup_{K=1}^{n+1}\tilde W_{n,K}$$ and next show that, for $n\ge 14$,
\begin{equation}\label{haudist}
d_{\mathcal H}\big(W_n, W+x_0\big)\le \frac{1}{n}.
\end{equation} 

Denote by $q_1=x_0, \ldots, q_{n+1}$ the points in the side $L_2+x_0$ of $W+x_0$ that are equidistributed at distance $1/n$, starting from $x_0$. By (\ref{equ}), $$\|a_2-q_2\|\le 3\epsilon +2\eta+\frac{\eta}{n}+3\epsilon$$ 
and inductively, for any $K=1, \ldots, n+1$,
$$\|a_K-q_K\|\le (K-1)(6\epsilon+(2+\frac{1}{n})\eta).$$
Then, by (\ref{distan}), 
\begin{equation}\label{7e}
d_{\mathcal H}\bigl(\tilde W_{n,K},L_1+q_K\bigr)\le 3\epsilon+\eta+n(6\epsilon+(2+\frac{1}{n})\eta)\le 7n\epsilon.
\end{equation}
Also,   
\begin{equation*}\label{triangular}
d_{\mathcal H}\Bigl(\bigcup_{K=1}^{n+1}L_1+q_K, W+x_0\Bigr)\le \frac{1}{2n},
\end{equation*}    
and consequently, by the  triangle inequality and since taking union over $K$ preserves the bound on the distance in (\ref{7e}), 
\begin{align*}d_{\mathcal H}\bigl(W_n,W+x_0\bigr)&\le 
d_{\mathcal H}\Bigl(\bigcup_{K=1}^{n+1}\tilde W_{n,K},\bigcup_{K=1}^{n+1} L_1+q_K\Bigr) +d_{\mathcal H}\Bigl(\bigcup_{K=1}^{n+1} L_1+q_K,W+x_0\Bigr)\\
&\le 7n\epsilon+\frac{1}{2n}\le \frac{1}{n}
\end{align*}
and (\ref{haudist}) follows.

Observe also that, by (\ref{7e}) and since $L_1+q_K\subset W+x_0$, we also get
$W_n\subset [W+x_0]_{7n\epsilon},$
which says that $W_n$ is a flat set. 

\subsubsection*{{\bf Construction of $\mathbf{W_n}$ for $\mathbf{s>2}$.}} As before, we construct inductively a path of points of the form $g_j\circ h_j(x_0)$ that approximates $W+x_0$. The construction is made in $n+1$ steps.  

Now denote by $W^{s-1}$ the $(s-1)$-parallelotope generated by $\omega^1, \ldots, \omega^{s-1}$. For $1\le l<s$, define $b^{l}=x_0+\omega^{{1}}+\ldots+\omega^{l}$,  that are vertices of $W^{s-1}+x_0$.

Assume we have constructed inductively a `flat' subset 
\begin{equation}\label{flatt}
\tilde W_{n}\subset \big[W^{s-1}+x_0\big]_{(7n)^{s-1}\epsilon}
\end{equation} 
of points of the form $g_j\circ h_j(x_0)$ satisfying 
\begin{equation}\label{distm-1}
d_{\mathcal H}\big(\tilde W_n, W^{s-1}+x_0\big)\le \frac{s-1}{n},
\end{equation}
and such that its last point, say $b=g_P\circ h_P(x_0)$, verifies $$\|b-b^{s-1}\|\le(7n)^{s-1}\epsilon.$$

In the first step of the construction of $W_n$,  we have the set $\tilde W_{n,1}:=\tilde W_n$ and start going back from $b$, through points close to $b^{s-2}, \ldots, b^2$ and $b^1$, to reach a point $\tilde a_1$ close to $a_1:=x_0$. 
This is done as follows. For $j>P$, define $g_j$ and $h_j$ as in (\ref{gjhj}) but choosing the sequences $\alpha_k, \beta_k\in\mathcal I^\ast$ so that (\ref{increase}) and (\ref{cone}) hold for the direction $-\omega^{s-1}$ for $j>P$. 
Then, $g_j\circ h_j(x_0)\in\mathcal C_b\bigl(B(-\omega^{s-1},\eta)\bigr)$ and we pick the largest $j_1>P$ such that 
$$g_{j_1}\circ h_{j_1}(x_0)\in C_b\bigl(B(-\omega^{s-1},\eta)\bigr)\cap B(b,1).$$  
Then, by (\ref{flatt}), 
$$\|g_{j_1}\circ h_{j_1}(x_0)-b^{s-2}\|\le (7n)^{s-1}\epsilon+3\epsilon+\eta.$$ 
Repeating the above procedure, but replacing $b$ by $g_{j_1}\circ h_{j_1}(x_0)$ and $-\omega^{s-1}$ by $-\omega^{s-2}$, we obtain the point $g_{j_2}\circ h_{j_2}(x_0)$ that is close to $b^{s-3}$. Continuing in this way, for $J=j_{s-1}$,  we obtain the point $\tilde a_1=g_J\circ h_J(x_0)$ satisfying
$$\|\tilde a_1-a_1\|\le (7n)^{s-1}\epsilon+(s-1)(3\epsilon+\eta).$$

Next, from $\tilde a_1$, we move in the direction $\omega^{s}$, defining accordingly $g_j$ and $h_j$, for $j>J$. Let $J'$ be the largest integer such that $g_j\circ h_j(x_0)\in C_{\tilde a_1}\bigl(B(\omega^{s},\eta)\bigr)\cap B(\tilde a_1, 1/K)$ and put $a_2=g_{J'}\circ h_{J'}(x_0)$. Finally, from $a_2$  we continue with the inductive definition of $g_j$ and $h_j$, $j>J'$, repeating a construction like the one for $\tilde W_n$ to obtain $\tilde W_{n,2}$. In this way, we obtain the sets $\tilde W_{n,K}$,
so that 
\begin{equation}\label{flattt}
\tilde W_{n,K}\subset \big[W^{s-1}+a_K\big]_{(7n)^{s-1}\epsilon}
\end{equation}
and
\begin{equation}\label{ecu}
d_{\mathcal H}\big(\tilde W_{n,K}, W^{s-1}+a_K\big)\le \frac{s-1}{n}
\end{equation}
We define
$$W_n=\bigcup_{K=1}^{n+1}\tilde W_{n, K}.$$

As before, we show $W_n\to_{d_{\mathcal H}} W+x_0$, where $W=W^s$ is the $s$-parallelotope generated by $\omega^1, \ldots, \omega^s$.  Consider the uniformly distributed points $q_1=x_0, \ldots, q_{n+1}$ at distance $1/n$, lying in the side of $W+x_0$ that contains $x_0$ and is parallel to $\omega^{s}$. From the inequality
$$\| a_K-q_K\|\le (K-1)\bigl((7n)^{s-1}\epsilon+(s-1)(3\epsilon+\eta)+\eta/n+3\epsilon\bigr)\le n\bigl((7n)^{s-1}+4s\bigr)\epsilon$$
we obtain, by (\ref{flattt}) and $n$ sufficiently large, 
$$\tilde W_{n,K}\subset [W^{s-1}+q_K]_{(7n)^{s-1}\epsilon+n\bigl((7n)^{s-1}+4s\bigr)\epsilon}\subset [W^{s-1}+q_K]_{(7n)^{s}\epsilon},$$
which gives 
$W_n\subset \big[W+x_0\big]_{(7n)^{s}\epsilon},$ since $W^{s-1}+q_K\subset W+x_0$. So $W_n$ is a flat set. Moreover, by (\ref{ecu}) and triangle inequality,
$$d_{\mathcal H}\big(\tilde W_{n,K}, W^{s-1}+q_K\big)\le\frac{s-1}{n}+(7n)^{s}\epsilon.$$
As  $$d_{\mathcal H}\Bigl(\bigcup_{K=1}^{n+1}W^{s-1}+q_K, W+x_0\Bigr)\le\frac{1}{2n},$$ another application of triangle inequality gives
\begin{equation}\label{cota}
d_{\mathcal H}\big(W_n, W+x_0\big)\le \frac{s-1}{n}+(7n)^{s}\epsilon+\frac{1}{2n}< \frac{s}{n},
\end{equation}
the last inequality holds for all sufficiently large $n$. Then we have $W_n\to_{d_{\mathcal H}} W+x_0$. 
 
Finally, we show that $W_n$ is image of points in $F$ through a similarity, that is, 
\begin{equation}\label{contention}
W_n\subset T_n(F),
\end{equation}  
for some similarity $T_n$. The proof is the same as the given in \cite{FHOR} but we reproduce it here for completeness. We show inductively that for any $x\in F$ and $l\in\mathbb{N}$, 
\begin{equation}\label{inductive}
\{g_j\circ h_j(x):j=0\ldots, l\}\subset\{g_{l}\circ S_{\beta}(x):\beta\in \mathcal I^\ast\},
\end{equation}
therefore (\ref{contention}) holds with  $x=x_0$ and $T_n:=g_p$, where $p$ is the largest $j$ with $g_j\circ h_j(x_0)\in W_n$. The case $l=1$ holds since $x=g_1\circ g_1^{-1}(x)$, where $g_1^{-1}=S_{\beta}$ for some $\beta\in\mathcal I^\ast$, and also  $h_1=S_{\beta'}$ with ${\beta'}\in\mathcal I^\ast$.  Assuming (\ref{inductive}) holds for $l-1$, then by definition (\ref{gjhj}), we get for $j=0,\ldots,l-1$,
$$g_j\circ  h_j=g_{l-1}\circ S_\beta=g_l\circ g_{l}^{-1}\circ g_{l-1}\circ S_\beta=g_l\circ S_{\beta''},$$
with $\beta''\in\mathcal I^\ast$.
Moreover, $h_l=S_{\beta{'''}}$ for some $\beta{'''}\in\mathcal I^{\ast}$, which completes the induction.

\subsubsection*{{\bf The sets $\mathbf{W_n(x)}$.}}
By construction, we write $W_n=\{g_j\circ h_j(x_0)\}_{j\in \mathcal T}$ for some finite set $\mathcal T$. Then, for any $x\in F$, define $W_n(x):=\{g_j\circ h_j(x)\}_{j\in \mathcal T}$.  Observe that $W_n(x)$ may not converge in the Hausdorff metric to $W+x$ since, although controlled by the bounds in (\ref{increase}), the size of the relative position between consecutive points, $$\|\Psi_j(x)\|=\|g_j\circ h_j(x)-g_{j-1}\circ h_{j-1}(x)\|,$$ for fixed $j$, may vary significantly with $x\in F$. However, a slight modification in the definition of the maps allows us to obtain 
\begin{equation}\label{equ1}
d_{\mathcal H}\big(W_n(x),W+x\big)\le \frac{s}{n}
\end{equation}
for any $x\in F$.
We proceed as follows. 

Let $0<\theta<\rho'$, where $\rho'$ as defined in Section \ref{section:prelims}. We modify definition (\ref{M}) by replacing $\rho'$ with $\theta$, that is, we pick $M$ such that $f_i^ m(x)\in B(a^i,{\theta})$ for all $x\in F$ and $m\ge M$.  Note that different values of $\theta$ may change the definition of the functions $g_j$ and $h_j$, and hence, the finite set $\mathcal T=\mathcal T_\theta$ may also change. However, the maximum of $\mathcal T_\theta$ is uniformly bounded in $0<\theta<\rho'$ (from the construction of $W_n$, this maximum can be estimated using the bounds (\ref{increase}) and property (\ref{cone}), which do not change with $\theta$). 
Let $$j^\ast=\max\{j: j\in \mathcal T_\theta, 0\le\theta\le\rho'\}.$$
Below we use this upper bound to choose an appropriate $\theta=\theta(\epsilon)$.

As before, choose $k_j$ and $m_j$ such that (\ref{fitpaso}) holds and put 
$\tilde x:=f_i^{m_j}\circ S_{\nu_j}\circ h_{j-1}(x).$ 
Then, recalling the identity (\ref{esta}),
\begin{equation*}
g_j\circ h_j(x)-g_{j-1}\circ h_{j-1}(x)=d_{j-1}c_{\nu_j}^{-1}c_{\gamma^{(i)}}^{-m_j}A_{j-1}O_{\nu_j}^{-1}O_{\gamma^{(i)}}^{-m_j}	\Phi_{i,k_j}(\tilde x),
\end{equation*}
and since $\rho\|D \Phi_{i, k_j}\|_{\mathrm{op}}\le 3\delta_{k_j}^{(i)}$,  we get by the first inequality in  (\ref{fitpaso}) that 
\begin{align*}
\bigl\|\Psi_j(x)-\Psi_j(x_0)\bigr\|&=d_{j-1}c_{\nu_j}^{-1}c_{\gamma^{(i)}}^{-m_j}\bigl\|\Phi_{i,k_j}(\tilde{x})-\Phi_{i,k_j}(\tilde{x}_0)\bigr\|\le 3\rho^{-1}\epsilon\|\tilde x-\tilde x_0\|.
\end{align*}
By definition, $\tilde x,\tilde x_0\in B(a^i,{\theta})$ if $x, x_0\in F$, 
so we obtain for all $x\in F$ and all $j$ that 
$$\bigl\|\Psi_j(x)-\Psi_j(x_0)\bigr\|\le  6\rho^{-1}\theta\epsilon,$$
and now we choose $\theta<\rho'$ so that 
\begin{equation*}\label{jotastar}
\bigl\|\Psi_j(x)-\Psi_j(x_0)\bigr\|\le \frac{\epsilon}{j^{\ast}}.
\end{equation*}

Considering the construction of $W_n$ with $\theta$ as above, then for any $j\in \mathcal T_\theta$ and $x\in F$ we get
\begin{equation}\label{bound relative}
\bigl\|\bigl(g_j\circ h_j(x)-x\bigr)-\bigl(g_j\circ h_j(x_0)-x_0\bigr)\bigr\|\le \sum_{q=1}^{j}\big\|\Psi_q(x)-\Psi_q(x_0)\big\|<\epsilon,
\end{equation}
which immediately gives $$d_{\mathcal H}(W_n(x)+(x_0-x), W_n)\le\epsilon.$$ 
This observation finally leads to 
\begin{align*}
d_{\mathcal H}\big(W_n(x), W+x\big)&=d_{\mathcal H}\big(W_n(x)+(x_0-x), W+x_0\big) \\ 
&\le d_{\mathcal H}\big(W_n(x)+(x_0-x), W_n\big)+d_{\mathcal H}\big(W_n, W+x_0\big) \\
&\le\epsilon + \frac{s-1}{n}+(7n)^{s}\epsilon+\frac{1}{2n}\\
&\le\frac{s}{n},
\end{align*}
where the second last inequality is the first inequality in (\ref{cota}), and the last inequality holds for all sufficiently large $n$ since $\epsilon=1/n^{s+2}$.

\begin{proof}[\bf Conclusion of the proof.]
Define $W_n(F)=\bigcup_{x\in F} W_n(x)$, so by  (\ref{equ1}) we have
\begin{equation}\label{queseyo}
W_n(F)\to_{d_{\mathcal H}}W+F \ \ \ \textrm{as } n\to\infty.
\end{equation}
Moreover, by (\ref{inductive}) we have $W_n(F)\subset T_n(F)$. 
Now consider $X:=[W+F]_1$, so that for all $n$ large enough, $W_n(F)\subset T_n(F)\cap X$  by (\ref{inductive}) and (\ref{queseyo}). By the compactness of the space of non empty compact subsets of $X$ with respect to the Hausdorff metric, the sequence  $\{T_n(F)\cap X\}_n$ has at least one limit point $\hat F$. Then, $\hat F$ is a weak tangent of $F$, and by (\ref{queseyo}), it contains $W+F$. 
\end{proof}

\section{Structural properties of $V_{\mathcal S}$ and proof of the main result}\label{section:structure}	

Applying the techniques from the proof of Theorem \ref{teotangent} of the previous section, we obtain the following structural properties of the set of overlapping directions.

\begin{theorem}\label{teovec}
Any unit vector in $V_\mathcal S$ is an overlapping direction.
Moreover, $V_{\mathcal S}$ is invariant under the group $G=\overline{\{O_\gamma:\gamma\in \mathcal I^\ast\}}$, that is, $Ov\in V_{\mathcal S}$ for any $O\in G$ and $v\in V_{\mathcal S}$. 
\end{theorem}  

\begin{proof}
Let $\omega^1,\ldots, \omega^m$ be linearly independent overlapping directions generating $V_{\mathcal S}$ and let $v=t_1\omega^1+\ldots+t_m\omega^m$. We assume $\|v\|=1$ and $t_i\ge 0$ for all $i$, with at least one of them positive; in case $t_i<0$ for some $i$, it is enough to consider the overlapping direction $-\omega^i$ instead of $\omega^i$. 

Pick any $x_0\in F$. Given $l$ large, by the construction from the proof of Theorem \ref{teotangent} and since $t_i\ge0$, there exists $n$ large enough, that we assume $l<n$, so that 
$$\Big(\mathcal C_{x_0}\big(B(v,1/l)\big)\cap A(x_0, 1/l)\Big)\cap W_n(x_0)-\{x_0\}\neq\{\emptyset\},$$ 
where $A(x,r):=B(x,2r)-B(x,r)$. This means that there are $j_l>0$ and maps $g_{j_l}=S_{\alpha_l}^{-1}$ and $h_{j_l}=S_{\beta_l}$, for some $\alpha_l, \beta_l\in \mathcal I^{\ast}$, such that 
$$g_{j_l}\circ h_{j_l}(x_0)\in\mathcal C_{x_0}\big(B(v,1/l)\big)$$
and moreover,
$$1/l<\|g_{j_l}\circ h_{j_l}(x_0)-x_0\|\le 2/l.$$ 
We define $\Phi_l=S_{\alpha_l}^{-1}\circ S_{\beta_l}-I=g_{j_l}\circ h_{j_l}-I$, so that for any $x\in F$ we have
by (\ref{bound relative}), 
\begin{equation*}\label{5am}
\|\Phi_l(x_0)-\Phi_l(x)\|<\epsilon.
\end{equation*}
Then, recalling that $\epsilon=1/n^{m+2}$, a simple geometrical argument shows that 
$$g_{j_l}\circ h_{j_l}(x)\in\mathcal C_{x}\big(B(v,1/l+1/n^m)\big)$$ and hence 
\begin{equation}\label{a}
\Phi_l(x)/\|\Phi_l(x)\|\to v \ \textrm{ as } \  l\to\infty
\end{equation}
for any $x\in F$.
We also obtain
$\|\Phi_l(x)\|\le 2/l+\epsilon$, and hence $0<\|\Phi_l\|_{{\infty}}\to 0$ as $l\to\infty$ since $F\Subset \rr^d$. Therefore, applying Lemma \ref{lemmaFraser}, we obtain a subsequence of $\Phi_l$ and some $a\in F$ such that the technical condition (\ref{boundderivative}) holds. But $\Phi_k(a)/\|\Phi_k(a)\|\to v$ since (\ref{a}) holds for any $x\in F$, therefore $v$ is an overlapping direction. 

Now we show the invariance with respect to $G$. Firstly, suppose that $O_\gamma\in G$ is the orthogonal part of $S_\gamma$, for some $\gamma\in \mathcal I^\ast$. Pick an overlapping direction $\omega$. Then, by definition, there are $a\in F$, $\rho>0$ and a pair $\alpha_k, \beta_k\in \mathcal I^\ast$ for each $k$ such that the maps $\Phi_k=S_{\alpha_k}^{-1}\circ S_{\beta_k}-I$ satisfy $\Phi_k(a)/\|\Phi_k(a)\|\to \omega$ as $k\to\infty$, and also the technical condition (\ref{boundderivative}) holds. By the argument from the first part of the proof, we assume that $\Phi_k/\|\Phi_k(x)\|\to	\omega$ as $k\to\infty$ for any $x\in F$. Then, applying Lemma \ref{lemmaFraser} with $\Gamma:= S_\gamma(F)\Subset\rr^d$, we further assume $a=S_\gamma(z)$ for some $z\in F$. 

From the identity
$$\bigl(S_{\alpha_k\gamma}^{-1}\circ S_{\beta_k\gamma}-I\bigr)(x)=
c_\gamma^{-1}O_{\gamma}^{-1}\Phi_k(S_\gamma(x)),$$
it follows that $0<\|S_{\alpha_k\gamma}^{-1}\circ S_{\beta_k\gamma}-I\|_{{\infty}}\to0$ and $\Phi_{\alpha_k\gamma, \beta_k\gamma}(z)/\|\Phi_{\alpha_k\gamma, \beta_k\gamma}(z)\|\to O_\gamma^{-1}\omega$. Also, 
$$\|D \Phi_{\alpha_k\gamma,\beta_k\gamma}\|_{\mathrm{op}}=c_\gamma^{-1}\|D(O_\gamma^{-1}\circ\Phi_{k}\circ S_\gamma)\|_{\mathrm{op}}=\|D\Phi_{k}\|_{\mathrm{op}}$$ 
hence, if $x\in B(z,c_\gamma^{-1}\rho)$, we get by (\ref{boundderivative}) 
\begin{equation}\label{pp}
c_\gamma^{-1}\rho\|D \Phi_{\alpha_k\gamma,\beta_k\gamma}\|_{\mathrm{op}}\le c_\gamma^{-1}\|\Phi_{k}(S_\gamma(x))\|= \|\Phi_{\alpha_k\gamma,\beta_k\gamma}(x)\|.
\end{equation}
Consequently, $O_\gamma^{-1}\omega$ is an overlapping direction and therefore, $V_{\mathcal S}$ is invariant under the inverse orthogonal map $O_\gamma^{-1}$.
But then, 
if $\{u_1, \ldots, u_m\}$ is an orthonormal basis of $V_{\mathcal S}$, we obtain that $\{O_\gamma^{-1}u_1,\ldots, O_{\gamma}^{-1}u_m\}$ is also an orthonormal basis of $V_{\mathcal S}$, so $\omega=a_1O_\gamma^{-1}u_1+\ldots+a_mO_\gamma^{-1}u_m$ for some scalars $a_i$, and in consequence $O_\gamma \omega=a_1u_1+\ldots+a_mu_m\in V_{\mathcal S}$, so $V_{\mathcal S}$ is invariant under $O_\gamma$.

Finally, for an arbitrary $O\in G$, pick a sequence $O_{\gamma^k}\to O$, with $O_{\gamma^k}\in\{O\gamma:\gamma\in \mathcal I^\ast\}$. Then $O_{\gamma^k}\omega\to O\omega$, and $O_{\gamma^k}\omega\in V_{\mathcal S}$ for all $k$, so $O\omega\in V_{\mathcal S}$ since $V_{\mathcal S}$ is closed.
\end{proof}

We have the following consequence, that will be useful later.

\begin{corollary}\label{Coro:invariance}
Let $\mathcal S$ be any IFS for the self-similar set $F$ and let $G$ be defined as above. Then, $V_{\mathcal O}$ is invariant under $G$. 
\end{corollary}
\begin{proof}
Let $\mathcal S'$ be an IFS for $F$ such that $V_{\mathcal S'}=V_{\mathcal O}$. Then, $\tilde{\mathcal S}=\mathcal S\cup \mathcal S'$ is also an IFS for $F$, and denoting by $\tilde G$ the corresponding group, then $G\subset \tilde G$, and the result follows from Theorem \ref{teovec} applied to $\tilde{\mathcal S}$.
\end{proof}

\subsection{Proof of Theorem \ref{paththeorem} }
Now we are in position to prove the main theorem, which we deduce from Theorem \ref{teotangent} and Theorem \ref{teovec}.

\begin{proof}[Proof of Theorem \ref{paththeorem}.] Given any $s$-dimensional subspace $V\subset V_{\mathcal S}$, first we  prove formula (\ref{formula}), i.e., that $$\dim_A F=\dim V+\dim_A(\pi_{V^{\perp}}(F)).$$

Let $\mathcal B=\{\omega^1,\ldots, \omega^s\}$ be a basis of $V$ with unit vectors, which are overlapping directions by Theorem \ref{teovec}. As above, we denote by $W$ the corresponding $s$-parallelotope. Extending $\mathcal B$ appropriately to a base of $\rr^d$, we further assume that $W\subset\rr^s\times\{\mathbf 0^{d-s}\}$, where $\mathbf 0^{d-s}$ is the zero element of $\rr^{d-s}$. For this reason we consider $W$ as a subset of $\rr^d$ or $\rr^s$, depending on whether we consider sumsets or Cartesian products.  

By Theorem \ref{teotangent}, there is  a  weak tangent $\hat F$ of $F$ containing $W+F$, so the monotonicity of the Assouad dimension and (\ref{boundtangent}) imply
\begin{equation}\label{uno}
\dim_A(W+F)\le\dim_A \hat F\le  \dim_A F=\dim_A (F+\mathbf 0)\le\dim_A (W+F).
\end{equation}
Then, to prove formula (\ref{formula}) we consider the set $W+F$, which for this purpose has a simpler structure than $F$.

Observe that 
\begin{equation*}
W+F\subset [W]_{\delta}\times \pi_{V^{\perp}}(F);
\end{equation*}
here the neighbourhood $[W]_{\delta}$ is considered in $\rr^s$ and $\delta$ is a constant such that $\pi_{V}(W+F)\subset [W]_\delta$.
Then, by the upper bound for the Assouad dimension of Cartesian products (see \cite[Theorem 2.1]{Fr}), we obtain the upper bound
\begin{equation}\label{dos1}
\dim_A(W+F)\le\dim_A([W]_{\delta})+\dim_A(\pi_{V^{\perp}}(F))=\dim V+\dim_A(\pi_{V^{\perp}}(F)).
\end{equation}

For the lower bound, let $\xi>0$ be such that $W$ contains an $s$-dimensional ball of radius $\xi$, and pick $n$ such that $\diam F_\alpha\le\xi/3$ for all $\alpha\in\mathcal I^n$; here $F_\alpha=S_{\alpha}(F)$. Note that 
\begin{equation*}
\dim_A (\pi_{V^{\perp}}(F))=\dim_A(\pi_{V^{\perp}}(F_{\tau}))
\end{equation*}
for some $\tau\in\mathcal I^n$ because of the finite stability of the Assouad dimension. (Indeed, the above equality holds for any $\alpha\in\mathcal I^n$ and for any $n$, as it is shown in Lemma \ref{lemma:projection} in Section \ref{section:weaknoweak}.)
Let $B_d(x,\xi/3)$ be a ball (in $\rr^d$) containing $F_{\tau}$. Note that if $B_s(a,\xi)$ is a ball (in $\rr^s$) contained in $\pi_{V}(W+x)$, then for any $z\in B_d(x,\xi/3)$ we have
\begin{equation}\label{dos}
B_s(a,\xi/3)\subset \pi_{V}(W+z),
\end{equation} and consequently
 \begin{equation}\label{tres}
 B_s(a,\xi/3)\times \pi_{V^{\perp}}(F_{\tau})\subset W+F_{\tau}\subset  W+F.
 \end{equation}
Now, recall that a lower bound for the Assouad dimension of the Cartesian products is (see \cite[Theorem 2.1]{Fr})
$$\dim_L X+\dim_A Y\le\dim_A(X\times Y),$$
where $\dim_L$ denotes the lower dimension, whose definition is dual to the Assouad dimension, and also $\dim_L B=s$ for any non empty open ball in $\rr^s$. Applying this formula in (\ref{tres}) together with the monotonicity of the Assouad dimension, we get by (\ref{uno}) and (\ref{dos1}),
$$
\dim_A F=\dim V+\dim_A(\pi_{V^{\perp}}(F))
$$
which is formula (\ref{formula}).

Finally, we prove the last statement of Theorem \ref{paththeorem}, that is, $\dim V<\dim_A F$ if $\dim V<d$. Denote by $\mathcal V(n)$ the class of $n$-dimensional vector subspaces of $\rr^d$. A set is $r$-separated if the distance between any two points in the set is at least $r$. We use the following.

\begin{lemma}\label{lemmaunifperf}
Let  $F\Subset\rr^d$ be a self-similar set. There are positive constants $t$ and $C$ such that the following holds. For any $x\in F$, $0<r<R\le{\rm diam}\: F$ and any $U\in\mathcal V(n)$, with $1\le n\le d$, then the orthogonal projection $\pi_U (F\cap B(x,R))$ contains at least $C(R/r)^t$ points that are $r$-separated.
\end{lemma}

\begin{proof}
Since $F\Subset\rr^d$, we pick $x_0,\ldots, x_d$  points in $F$ so that the vectors $v_j=x_j-x_0$, $1\le j\le d$, span $\rr^d$. Then, there is some $C_1>0$ such that for any $U\in \mathcal V(n)$ we have 
\begin{equation}\label{uuu}
\| \pi_U(v_j)\|\ge C_1, \ \ \textrm{for some} \ \ 1\le j\le d.
\end{equation}   
In fact, given any unit vector $u\in U$, then $\| \pi_U(v_j)\|\ge|\pi_u(v_j)|=|u\cdot v_j|$, and hence (\ref{uuu}) is an easy consequence of the continuity of the non vanishing function $f(u)=\sum_{j=1}^d|u\cdot v_j|$ in $S^{d-1}$.

Given $x\in F$ and $0<\tilde r\le\diam F$, let $i\in \mathcal I^\ast$ be such that 
$$F_i\subset B(x,\tilde r) \ \ \textrm{but  } \ F_{\overline i}\setminus B(x,\tilde r)\neq \emptyset ,  \ \ \textrm{with }x\in F_{i};$$
here ${\overline i}=(i_1, \ldots, i_{k-1})$ if $i=(i_1,\ldots, i_k)$.  Note that $(c^{\ast})^{-1}\tilde r\le c_i$, with $c^{\ast}$ the maximum of the contraction ratios of the similarities of the IFS. Then, by linearity of the projection, 
\begin{align*}
\|\pi_U(S_i(x_j))-\pi_U(S_i(x_0))\|
=c_i\|\pi_U( O_{i}(v_j ))\|
&=c_i\|\pi_{O_i^TU}(v_j)\|,
\end{align*}
and hence, by (\ref{uuu}), there is some $j_0$ for which  
$$\|\pi_U(S_i(x_{j_0}))-\pi_U(S_i(x_0))\|\ge c_i C_1\ge C_3 \tilde r,$$ with $C_3=C_1(c^{\ast})^{-1}$. This shows that for $x\in F$ and $0<\tilde r\le\diam F$ there are points $y_1, y_2\in \pi_U(B(x,\tilde r)\cap F)$ such that $\|y_1-y_2\|\ge C_3\tilde r$.

Now, given $0<r<R\le\diam F$ and $x\in F$, the above property may be used iteratively to produce an $r$-separated subset of $\pi_U(B(x, R)\cap F)$ with cardinality at least $C(R/r)^t$, with  $C$ a positive constant (independent of $r$ and $R$) and $t=\log2/\log(8/C_3)$. In fact, starting with $\tilde r=R/2$, we get two points $y_1, y_2\in\pi_U(B(x, R/2)\cap F)$ that are at least $C_3R/2$ apart. Let $z_l\in B(x, R/2)\cap F$, $l=1,2$, be such that $\pi_U(z_l)=y_l$. Then, in each ball $\pi_U(B(z_l,C_3R/8)\cap F)\subset\pi_U(B(x,R)\cap F)$ we find points $y_{l,1}, y_{l,2}$ that are at least $C_3^2/16$ apart. Continuing in this way, after $k$ iterations we find $2^k$ points in $\pi_U(B(x,R)\cap F)$ that are $C_3^kR/8^k$ apart.
Choosing $k$ so that $r\approx C_3^kR/8^k$, we obtain $(R/r)^t\approx 2^k$ since $t=\log2/\log(8/C_3)$.
\end{proof}

Now, assume the strict inequality $\dim V<d$ holds. Given $y\in \pi_{V^{\perp}}(F)$, pick $x\in F$ such that $y=\pi_{V^{\perp}}(x)$. Then, for $0<R\le\diam \pi_{V^{\perp}}(F)$ we have
$$\pi_{V^{\perp}}(F)\cap B(y,R)\supset\pi_{V^{\perp}}(F\cap B(x,R)).$$
Then, Lemma \ref{lemmaunifperf} immediately implies $\dim_A( \pi_{V^{\perp}}(F))\ge t>0$, so by formula (\ref{formula}) we obtain $\dim V<\dim_A F$, which concludes the proof. 
\end{proof}

\section{Examples.}\label{sectionexamples}

\subsection{Homogeneous self-similar sets.}
Recall that an IFS is homogeneous (or equicontractive) if all its defining similitudes have the same contraction ratio.

\begin{proposition}
Let $F\Subset \rr^d$ be a self-similar set that is the attractor of an homogeneous IFS, with contraction ratio $c$, that does not satisfy the WSP. If the associated orthogonal group contains only the identity, then the set of overlapping directions is \begin{equation}\label{ome}
\Omega_{\mathcal S}=\Bigl\{\textrm{limit points of } \left\{\frac{S_{\beta_k}(0)-S_{\alpha_k}(0)}{\|S_{\beta_k}(0)-S_{\alpha_k}(0)\|}\right\}_k : |\alpha_k|=|\beta_k|  \ \textrm{and } 0<\frac{\|S_{\beta_k}(0)-S_{\alpha_k}(0)\|}{c^{|\beta_k|}}\underset{k}{\to}0\Bigr\}.
\end{equation}
\end{proposition}
\begin{proof}
By the hypothesis of trivial orthogonal part of the similitudes, we have
\begin{align}\label{ah}
\Phi_k(x)&=S_{\alpha_k}^{-1}S_{\beta_k}(x)-x=(c_{\alpha_k}^{-1}c_{\beta_k}-1)x+c_{\beta_k}^{-1}(S_{\beta_k}(0)-S_{\alpha_k}(0)).
\end{align} 
Let $A$ be the right hand side of (\ref{ome}). Clearly, by the homogeneity assumption and the above identity, $A\subset\Omega_{\mathcal S}$; the technical condition (\ref{boundderivative}) holds since $\Phi_k$ is constant. On the other hand,  if $\omega\in\Omega_{\mathcal S}$, then $0<\|\Phi_k\|_\infty\to0$, so (by picking $x=(0,\ldots,0) $ in (\ref{ah})) necessarily $c_{\beta_k}^{-1}\|S_{\beta_k}(0)-S_{\alpha_k}(0)\|\to0$ as $k\to\infty$, which implies $(c_{\alpha_k}^{-1}c_{\beta_k}-1)\to0$ as $k\to\infty$. Then, from the equicontractivity assumption, it follows easily that $c_{\alpha_k}=c_{\beta_k}$, and hence $|\alpha_k|=|\beta_k|$ for all $k$ large enough. Then, $\Phi_k$ is constant for all large enough $k$, so the technical condition (\ref{boundderivative}) trivially holds and in consequence, (\ref{ome}) follows.  
\end{proof}

\begin{example}\label{Example}
This set of examples are a modification from the example in \cite[Section 4.1]{FHOR}, which in turn are based on an example from Bandt and Graf \cite[Section 2 (5)]{BG}. 
Consider the similarity maps on $[0,1]^2$ given by 
$$S_1(x)=x/5, \ \ \ S_2(x)=x/5+(4/5,0) \ \ \ \textrm{and }\ \ S_3(x)=x/5+(0,4/5).$$ 
Also, for $$t=4\sum_{k\ge0}5^{-2^k}=0.9664 \ldots, \ \textrm{ and } \  \   u=4-4\sum_{k\ge1}5^{-2^k}=3.8335\ldots ,$$ 
let
$$S_4(x)=x/5+(t/5,0), \ \ \ S_5(x)=x/5+(0,u/5) \ \ \ \textrm{and } \ \ S_6(x)=x/5+(1/2, u/5).$$

{\rm(a)} Consider the IFS $\mathbf{\mathcal S_1}=\{S_1, S_2, S_3, S_4\}$, with attractor $F_1$; see Figure \ref{fig:2-1}. This example was considered in \cite[Section 4.1]{FHOR} as an example of a planar self-similar set which is the attractor of an IFS with no WSP and such that $\dim_H F_1<1\le\dim_A F_1<2$.
The parameter $t$ is chosen so that  $\{S_1, S_2, S_4\}$ is an IFS with no WSP, and hence producing overlaps on the horizontal direction;  see \cite{BG} and \cite{FHOR} for details. Moreover, the upper bound $\dim_A F_1\le1+ \log2/\log5$ is given in the later paper: if $\pi_x$ and $\pi_y$ denote the orthogonal projection onto the $x$- and $y$-axes, respectively, then $F_1\subset\pi_x(F_1)\times\pi_y(F_1)$, where $\pi_x(F_1)$ results a self-similar set with no WSP, so $\dim_A(\pi_x(F_1))=1$,  while $\pi_y(F_2)$ is a self-similar set that is the attractor of an IFS with the open set condition consisting of two maps with contraction ratio $1/5$, so $\dim_A(\pi_y(F_2))=\log2/\log5$.  

As we mentioned above, by construction,  $(1,0)$ is an overlapping direction, and if  $V=\text{gen}\{(1,0)\}$, then $\pi_{V^{\perp}}=\pi_y$. Hence we obtain $\dim_A F_1=1+ \log2/\log5$ from (\ref{formula}) in Theorem \ref{paththeorem}.

Observe that in this case we can conclude, a posteriori, that $V_{\mathcal O}=V$ since $\dim_A F_1<2$.

{\rm(b)} Let $\mathbf{\mathcal  S_2}=\{S_1, S_2, S_3, S_4, S_5\}$, with attractor $F_2$; see Figure \ref{fig:2-2}. As above, the maps $S_1$, $S_2$ and $S_4$ generates the overlapping direction $(1,0)$, while the parameter  $u$ is chosen so that $\{S_1, S_3, S_5\}$ produce vertical overlaps, so they generate the overlapping direction $(0,1)$. Then, $V_{\mathbf{\mathcal S_2}}=\rr^2=V_{\mathcal O}$, and in consequence, $\dim_A F_2=2$.

{\rm(c)} Finally, let $\mathbf{\mathcal S_3}=\{S_1, S_2, S_3, S_4, S_6\}$ with attractor $F_3$; see Figure \ref{fig:2-3}. As above, $(1,0)$ is an  overlapping direction. Moreover, $\pi_y(F_3)$ is a self-similar set in $\rr$ given by the maps $f_1(x)=x/5$, $f_2(x)=x/5+4/5$ and $f_3(x)=x/5+u/5$, and by the choice of $u$, this IFS does not satisfy the WSP.  In consequence, $\dim_A \pi_y(F_3)=1$, so by (\ref{formula}) we get $\dim_A F_3=2$. 

In this example, we do not know if $V_{\mathcal O}=\text{gen}\{(1,0)\}$, that is, if $(0,1)$ is not an overlapping direction. If this is the case, then this would be a counterexample for  Question \ref{question:proy}.
\end{example}

\begin{figure}\label{eje1}\begin{center}
\subfigure[The first step in the construction and the self-similar set $F_1$.]{
\label{fig:2-1}
\includegraphics[height=4cm, width=4cm]{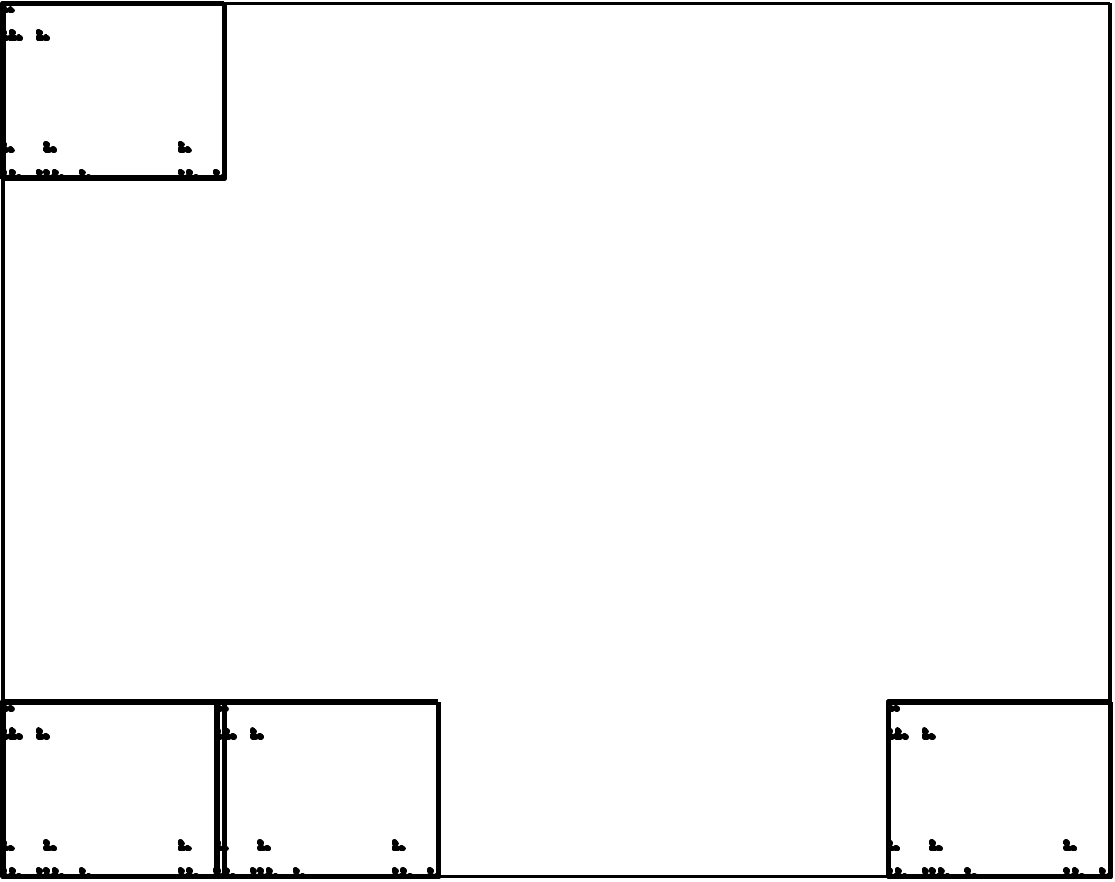}
\hspace{2cm}
\includegraphics[height=4cm, width=4cm]{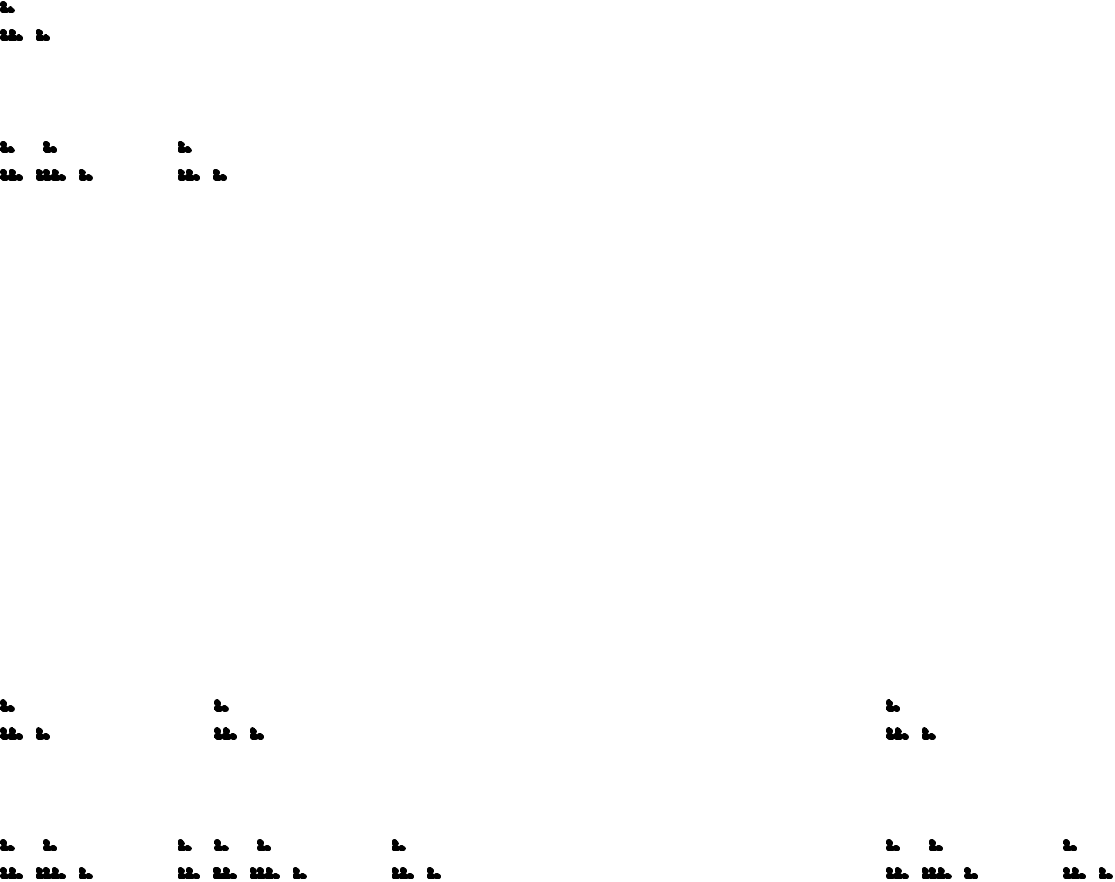}
}
\\
\subfigure[The first step in the construction and the self-similar set $F_2$.]{
\label{fig:2-2}
\includegraphics[height=4cm, width=4cm]{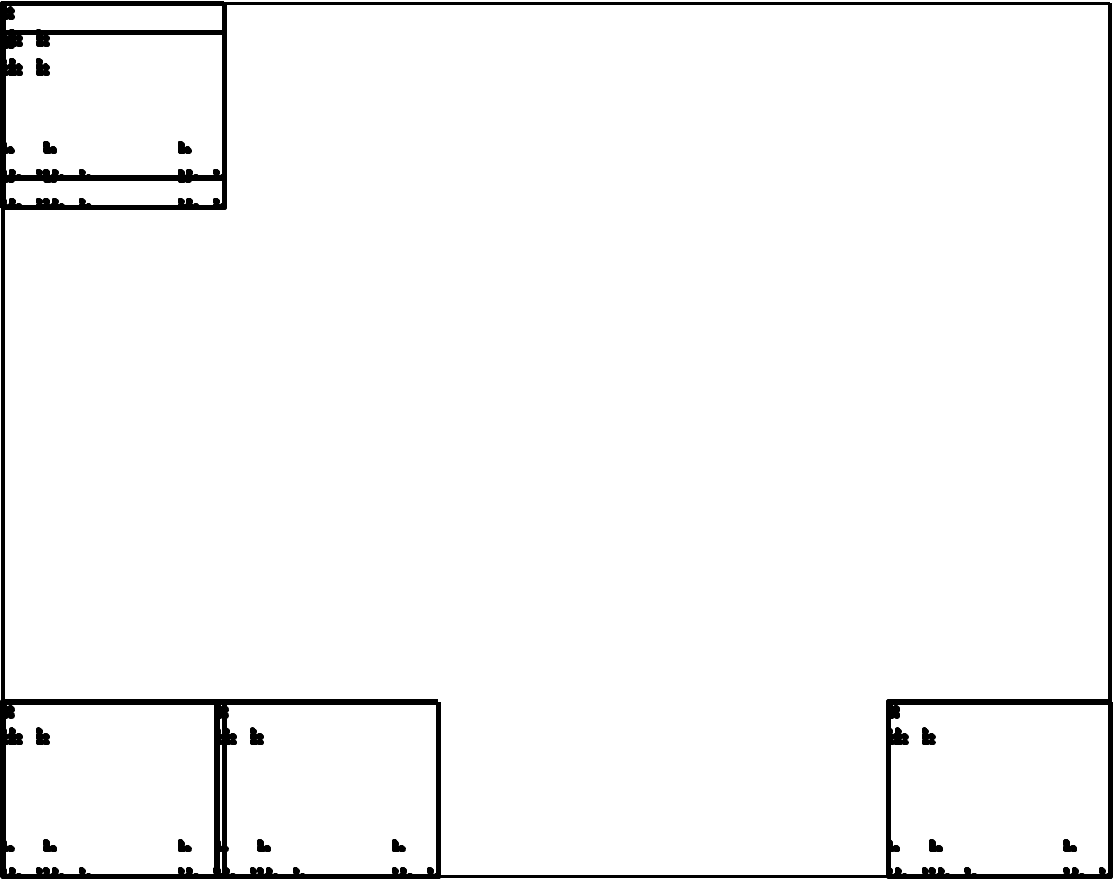}
\hspace{2cm} 
\includegraphics[height=4cm, width=4cm]{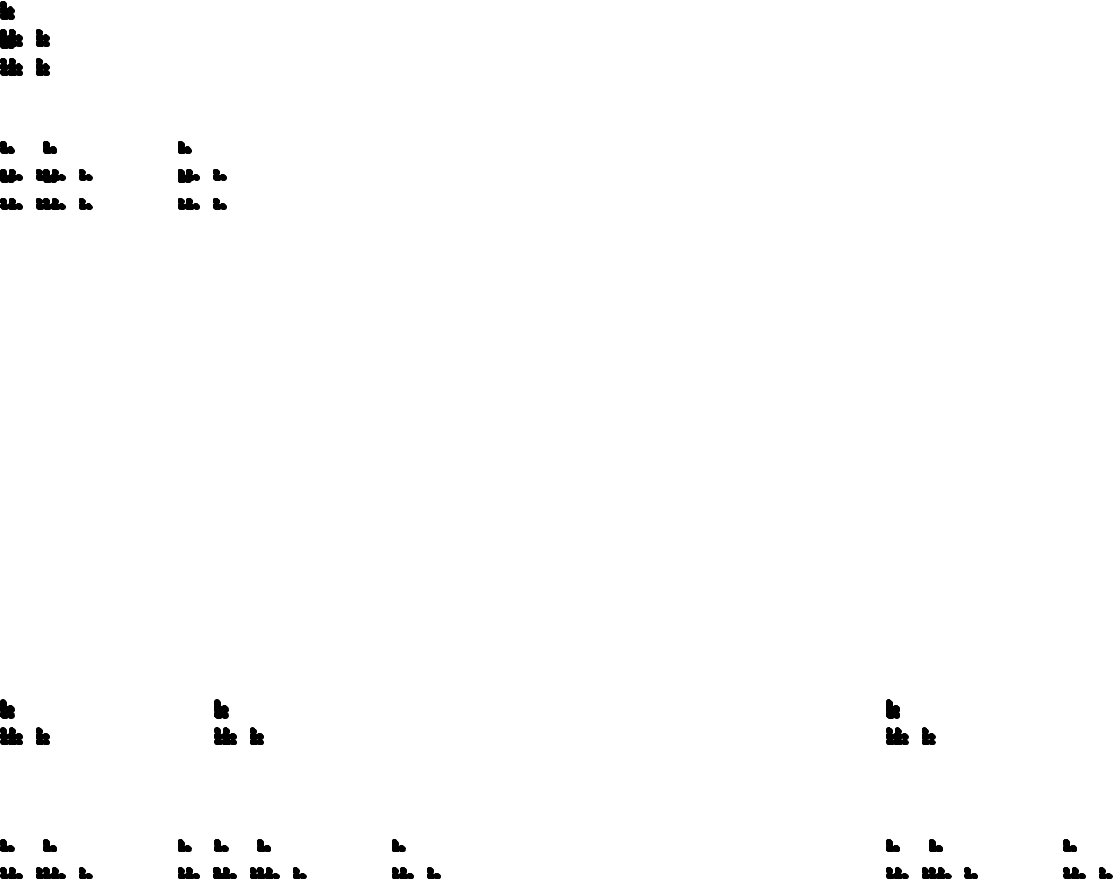}
}
\\
\subfigure[The first step in the construction and the self-similar set $F_3$.]{
\label{fig:2-3}
\includegraphics[height=4cm, width=4cm]{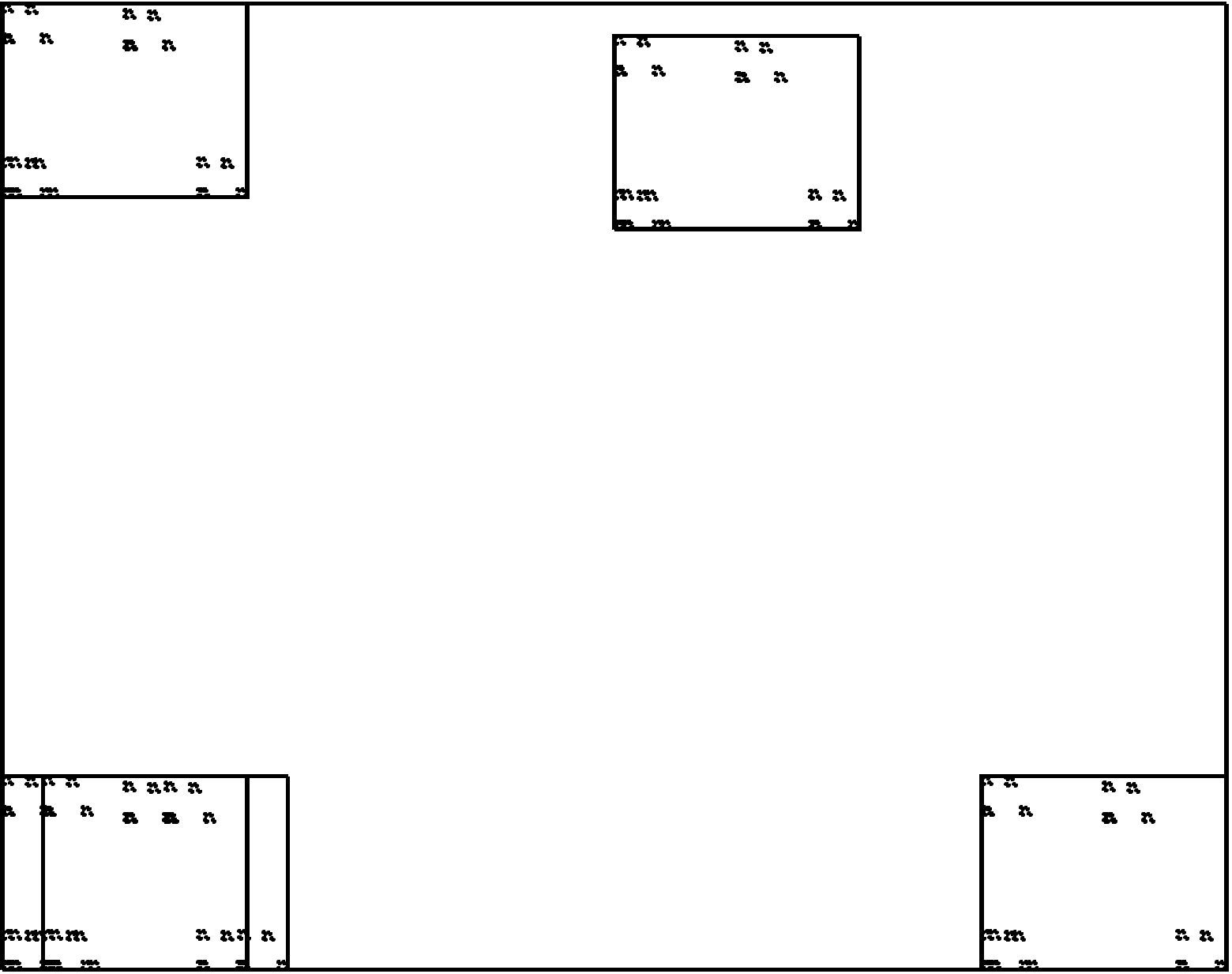}
\hspace{2cm} 
\includegraphics[height=3.9cm, width=4cm]{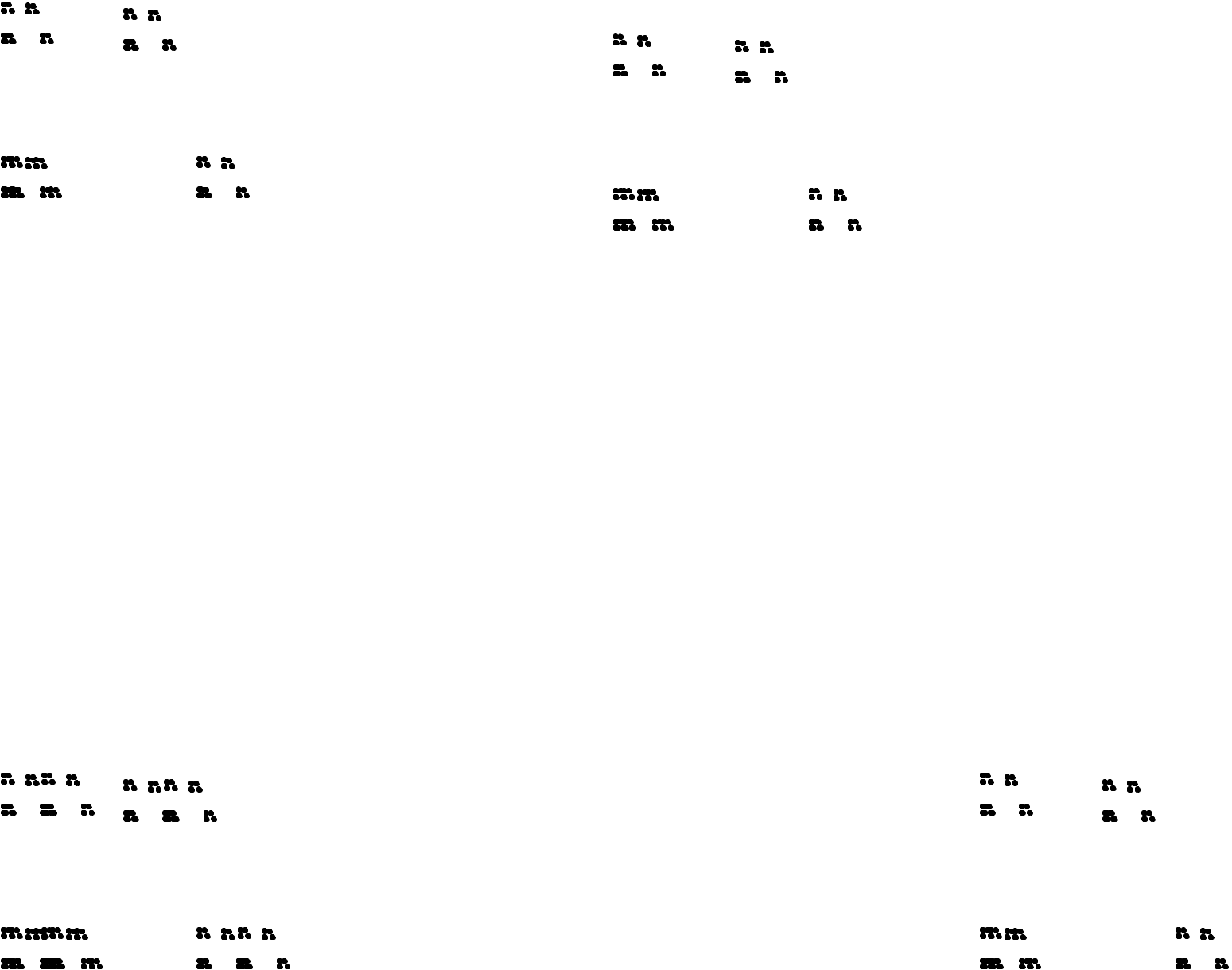}
}
\caption{}
\end{center}
\end{figure}

\section{Assouad dimension of overlapping graph directed self-similar sets}\label{section:graph}

Graph directed self-similar sets are a generalization of self-similar sets and were introduced by Mauldin and Williams in \cite{MW98}. Our results can be generalized to this setting.  

Let $\Gamma=(\mathcal E, \mathcal V)$ be a directed graph, that is, $\mathcal V=\{1,\ldots,q\}$ is a set of vertices and $\mathcal E$ is a finite set of directed edges, with initial an final points in the set $\mathcal V$, and such that for any $i\in\mathcal V$ there is an edge $\e\in\mathcal E$ starting from $i$. Let $\mathcal E_{i,j}\subset \mathcal E$ be the set of edges from $i$ to $j$.

For each edge $\e\in\mathcal E$, let $S_\e:\rr^d\to\rr^d$ be a contracting similarity, that has the form $S_\e=r_\e O_\e x+d_\e$, with $0<r_\e<1$, $O_\e$ an orthogonal matrix and $d_\e\in\rr^d$. Then, there is a unique $q$-tuple of nonempty compact sets such that
$$F_i=\bigcup_{j=1}^q\bigcup_{\e\in\mathcal E_{i,j}}S_\e(F_j).$$ 
The family $\mathcal S=\{S_\e : \e\in\mathcal E\}$ is a graph directed iterated function system of similarities (GDIFS) and $\{F_1,\ldots, F_q\}$ is the corresponding family of graph directed self-similar sets.

Let $\mathcal E_{i,j}^k$ be the set of finite sequences of edges $(\e_1,\ldots, \e_k)$ that form a path from $i$ to $j$, and let $\mathcal E_{i,j}^\ast= \bigcup_{k=1}^\infty \mathcal E_{i,j}^k$. 
We assume that $\Gamma$ is strongly connected, which means that for any pair of vertices $i,j\in\mathcal V$, the set $\mathcal E_{i,j}^k$ is nonempty for some $k$, which may depend on $i$ and $j$. In this case, the GDIFS $\mathcal S$ is also called strongly connected. This is a transitivity condition that ensures the same behavior for each of the sets $F_i$. In fact, under this assumption there is coincidence of the dimensions:
$$\dim_H F_i=\dim_H F_j=\dim_B F_j \ \ \ \textrm{for every } i,j\in\mathcal V;$$ 
this follows using implicit methods as in the proof of \cite[Corollary 3.5]{Fal}.

The weak separation property for GDIFS was defined by Das and Edgar in \cite{DE05}, where they adapted the many equivalent definitions given in \cite{Ze}. For $i,j\in\mathcal V$, let 
$$\mathcal F_{i,j}=\{S_{\e}^{-1}\circ S_{\ff}:\e,\ff\in\mathcal E_{i,j}^\ast\}.$$

Assume $F_i\Subset\rr^d$ for some $i\in\mathcal V$ (equivalently for all $i$, by strongly connectedness).  We say $\mathcal S=\{S_\e: \e\in\mathcal E\}$ satisfies the graph directed weak separation property  (GDWSP) if, for some (equivalently for all) $i\in\mathcal V$, the identity is an isolated point of $\mathcal F_{i,i}$ in the topology of pointwise convergence. 

\begin{definition}
Let $\Phi_k=S_{\e_k}^{-1}\circ S_{\ff_k}-I$. If the GDWSP is not satisfied, there are $\e_k, \ff_k\in\mathcal E_{i,i}^\ast$ such that $0<\|\Phi_k\|_\infty\to0$ as $k\to\infty$. We say that $\omega\in S^{d-1}$ is an $i$-overlapping direction for $\mathcal S$ if there are $a\in F_i$ and $\rho>0$ such that  $\omega=\lim_{k\to\infty}\Phi_k(a)/\|\Phi_k(a)\|$ and 
\begin{equation}\label{p}
\rho\|D\Phi_{k}\|_{\textrm{op}}\le\|\Phi_k(a_k)\|,
\end{equation} 
where $a_k$ is any point in $\textrm{argmin}\{\|\Phi_k(x)\|:x\in B(a,\rho)\}$.  Denote by $V_{\mathcal S, i}$ the set of scalar multiples of $i$-overlapping directions. 
\end{definition}

Observe that if the GDWSP is not satisfied, then $V_{\mathcal S, i}\neq\emptyset$ because of Lemma \ref{lemmaFraser}. 
The extension of the results from the self-similar case are summarized below.

\begin{theorem}
Let $\mathcal S$ be a strongly connected graph directed iterated function system with attractor $\{F_1,\ldots, F_q\}$. Assume $F_i\Subset \rr^d$.
\begin{enumerate}[1)]
\item If $\mathcal S$ satisfies the GDWSP, then for all $i\in\mathcal V$ we have $\hau^s(F_i)>0$ and $\dim_HF_i=\dim_AF_i$.
\item If $\mathcal S$ does not verify the GDWSP, then for each $i\in\mathcal V$ there is a weak tangent of $F_i$ containing $W_i+F_i$, where $W_i$ is the parallelotope associated to a maximal subset of linearly independent $i$-overlapping directions for $\mathcal S$. Moreover, $V_{\mathcal S,i}$ is a vector space, and given any vector subspace $V$ of $V_{\mathcal S,i}$, we have
$$\dim_A F_i=\dim V+\dim_A\pi_{V^\perp}F_i.$$
Moreover, $\dim V_{\mathcal S,i}<\dim_A F_i$ whenever $dim V_{\mathcal S,i}<d$.

Also, 
\begin{equation}\label{gd}
\dim V_{\mathcal S,i}=\dim V_{\mathcal S,j}
\end{equation} 
for $i,j\in \mathcal V$.
\end{enumerate} 
\end{theorem}

\begin{proof}
Part $1)$ is stated only for completeness; its proof is the same as in the one dimensional case, which is given in \cite[Theorem 4.2 (1)]{FO}.
 
For $2)$, a straightforward adaptation of the proofs of Theorems \ref{teotangent}, \ref{teovec} and \ref{paththeorem} show the assertions, with the exception of the last one; note that, when constructing the weak tangent, for each $i\in\mathcal V$ we have to consider the group $G_i=\overline{\{O_\e: \e\in\mathcal E_{i,i}^\ast\}}$. 

For the last assertion, given $i,j\in\mathcal V$, fix $\g\in\mathcal E_{j,i}^\ast$ and $\mathbf{h}\in\mathcal{E}^\ast_{i,j}$ and let $v$ be an $i$-overlapping direction. Firstly we show that $O_{\mathbf{h}}^{-1} v$ is a $j$-overlapping direction. In fact, let $\e_k, \ff_k\in\mathcal E_{i,i}^\ast$  be such that $\|\Phi_k\|\to0$ and $\Phi_k(a)/\|\Phi_k(a)\|\to v$ for some $a\in F_i$.  Then, $\tilde\e_k:=\g\e_k\h$ and $\tilde\ff_k:=\g\ff_k\h$ belong to $\mathcal E_{j,j}^\ast$ and
$$\tilde\Phi_k:=S_{\tilde\e_k}^{-1}\circ S_{\tilde\ff_k}-I=c_{\mathbf{h}}^{-1}O_{\mathbf{h}}^{-1}(\Phi_k\circ S_{\mathbf h}).$$
We assume further that $\Phi_k(x)/\|\Phi_k(x)\|\to v$ for any $x\in F_i$, which is justified as in the proof of Theorem \ref{teovec}. Then, since $S_{\mathbf{h}}(F_j)\Subset F_i$, Lemma \ref{lemmaFraser} allows us to assume $a=S_{\mathbf{h}}(z)$ for some $z\in F_j$. 
Then, $0<\|\tilde\Phi_k\|\to0$ and $\tilde\Phi_k(z)/\|\tilde\Phi_k(z)\|\to O_{\mathbf{h}}^{-1}v$. Also, (\ref{p}) follows as in (\ref{pp}), hence $O_{\mathbf{h}}^{-1}v$ is a $j$-overlapping direction. 

Finally, since the argument is symmetric in $i$ and $j$, then we get the inclusion of the spaces 
$$V_{\mathcal S,i}\subset O_{\mathbf{h}}V_{\mathcal S,j}\subset O_{\mathbf h}O_{\mathbf g}V_{\mathcal S,i},$$ so (\ref{gd}) holds.
\end{proof}

\section{Structure of overlapping/non overlapping self-similar sets}\label{section:weaknoweak}

Now we concentrate on the topological structure of the overlapping self-similar sets which also are attractors of non overlapping IFS. For example, the unit cube in $\rr^d$ is the attractor of an IFS $\mathcal S$ satisfying the weak separation condition (indeed, the open set condition) consisting of $2^d$ obvious similarities with contraction $1/2$, and besides, if we add to $\mathcal S$ a similarity with contraction ratio $1/2$ with an appropriate irrational translation, we have that the new IFS does not verify the weak separation property. Another example is the Cartesian product $Q\times C$, where $Q$ is the unit cube in $\rr^{n}$ for some $1\le n<d$, and $C\Subset\rr^{d-n}$ is an homogeneous self-similar set with contraction ration $1/K$ for some integer $K$ and satisfying the weak separation property. These examples share the property of having non empty interior (relative to $\rr^{n}$ in the last example), and below we show that this situation holds in general but imposing an extra separation condition.

Firstly, observe that for such self-similar sets in $\rr$, it is always the case that they have non empty interior, as shows the following proposition.

\begin{proposition}
Let $F\Subset\rr$ be a self-similar set that is both the attractor of an IFS satisfying the weak separation property and of an IFS not satisfying the weak separation property. Then $F$ has non empty interior.
\end{proposition}
\begin{proof}
Since $F$ is the attractor of an IFS not satisfying the WSP, then $\dim_A F=1$ by \cite[Theorem 3.1]{FHOR}. Also, since $F$ is the attractor of an IFS satisfying the weak separation property, we get $\dim_H F=\dim_A F=1$ by \cite[Theorem 2.1]{FHOR}. Then $F$ has non empty interior since it is the attractor in $\rr$ of an IFS satisfying the WSP with $\dim_H F=1$; see \cite[Theorem 3]{Ze}.
\end{proof}

Next we present an extension of the above result to the higher dimensional case, but we need to impose a stronger separation condition. Recall that an IFS $\{R_1,\ldots, R_{m'}\}$ verifies the open set condition (OSC) if there is a non empty bounded open set $U$ such that 
\begin{equation}\label{OSC}
\bigcup_{i=1}^{m'}R_i(U)\subset U, 
\end{equation}
with disjoint union. It is well known that the OSC is equivalent to the strong open set condition (SOSC), which further requires that the open set intersects the attractor; see \cite{S94}.

\begin{theorem}\label{Theorem:product-structure}
Let $F\Subset\rr^d$ be a self-similar set that is both the attractor of an IFS satisfying the open set condition and of an IFS not satisfying the weak separation property. Let $p=\dim_{O} F$. Then, there is a non trivial $p$-dimensional cube $Q\subset V_{\mathcal O}$ and a set $C\subset V_{{\mathcal O}}^\perp$, with $\dim_A C=\dim_A \pi_{V_{\mathcal O}^\perp} F$, such that $Q\times C\subset F$.
\end{theorem}

\begin{question}
It is possible to improve to weak separation property the hypothesis of open set condition? We believe this is the case.
\end{question}

Let $\{S_1,\ldots, S_m\}$ be an IFS for $F$ not satisfying the WSP and which also attains the overlapping dimension $p$ of $F$.  As in the proof of Theorem \ref{paththeorem}, we assume that the corresponding $p$-parallelotope $W$ is contained in $V_{\mathcal O}=\rr^p\times\{\mathbf{0}\}^{d-p}$. 

Before the proof we need two preliminary lemmas. The next lemma is stated for $V_{\mathcal O}$ but clearly is true for $V_{\mathcal S}$, for any IFS $\mathcal S$.

\begin{lemma}\label{lemma:projection}
For any $\alpha\in\mathcal{I}^\ast$ we have $\dim_A\pi_{V_{\mathcal O}^\perp}F_\alpha=\dim_A\pi_{V_{\mathcal O}^\perp}F$.
\end{lemma}
\begin{proof}
For easy of notation, let $V=V_{\mathcal O}$. For any $x\in\rr^d$, decompose $x=x_V+x_{V^\perp}$, with $x_V\in V$ and $x_{V^\perp}\in V^\perp$. Then, 
\begin{equation*}
S_\alpha(x)-S_\alpha(y)=c_\alpha O_\alpha(x-y)=c_\alpha O_\alpha(x_V-y_V)+c_\alpha O_\alpha(x_{V^\perp}-y_{V^\perp}).
\end{equation*} 
By Theorem \ref{teovec}, $V$ is invariant by $O_\alpha$ for any $\alpha\in\mathcal{I}^\ast$, and then it is easily seen that so is $V^\perp$. In consequence,
\begin{equation*}
\|\pi_{V^\perp} S_\alpha(x)-\pi_{V^\perp} S_\alpha(y)\|=c_\alpha\|x_{V^\perp}-y_{V^\perp}\|=c_\alpha\|\pi_{V^\perp}(x)-\pi_{V^\perp}(y)\|.
\end{equation*}
Then, any set of $r$-separated points in $\pi_{V^{\perp}}(F)$ (that is, any two points in the set are at distance at least $r$) is in one to one correspondence with a $c_\alpha r$-separated set in $\pi_{V^{\perp}}(F_\alpha)$, and consequently, it follows directly from the definition of the Assouad dimension that $\dim_A\pi_{V^{\perp}}(F)=\dim_A \pi_{V^{\perp}}(F)$. 
\end{proof}

\begin{lemma}\label{lemma:prod}
Given $r>0$ and $y\in W+F$, there are sets $Q$ and $C$ as in the statement of Theorem \ref{Theorem:product-structure} such that $Q\times C\subset (W+F)\cap B(y,r)$. 
\end{lemma}
\begin{proof}
The  argument is similar to the used to obtain (\ref{tres}) in the proof of Theorem \ref{paththeorem}. Again, define $V=V_{\mathcal O}$. Let $x\in F$ be such that $y\in W+x$. Pick a ball $B(u,\xi)\subset B(y,r)$ with $u\in W+x$ and such that $$B_p(a,\xi):=\pi_{V}(B(u,\xi))\subset \pi_V(W+x).$$
We deduce as in (\ref{dos}) that
\begin{equation*}
B_p(a,\xi/3)\subset \pi_{V}(W+z) \ \ \textrm{for any } z\in B(x,\xi/3).
\end{equation*}
Then, pick a cube $Q_u\subset B(u,\xi/3)$ centered at $u$ with sides parallel to the coordinate axes, and let $Q_x$ be the translation of $Q_u$ centered at $x$, so $\pi_{V^{\perp}}(Q_u)=\pi_{V^{\perp}}(Q_x)$. As $x\in F$, there is $\alpha\in\mathcal I^\ast$ such that $F_\alpha\subset Q_x$. In consequence, $\pi_V(Q_u)\times\pi_{V^{\perp}}(F_\alpha)\subset  (W+F)\cap B(y,r)$, where $\pi_V(Q_u)$ is a non trivial cube contained in $V$ and $\dim_A\pi_{V^{\perp}}(F_\alpha)=\dim_A \pi_{V^{\perp}}(F)$ by Lemma \ref{lemma:projection}.
\end{proof}

\begin{proof}[Proof of Theorem \ref{Theorem:product-structure}]
We use the notation from the proof of Theorem \ref{teotangent}, where the `pretangents' $W_n(F)$ satisfy 
\begin{equation}\label{subsets}
F\subset W_n(F)\subset T_n(F)
\end{equation}
with $T_n=S_{\beta_n}^{-1}$ for some $\beta_n\in\mathcal I^\ast$.
Moreover, recall the estimate (\ref{equ1})  $$d_{\mathcal H}(W_n(F), W+F)\lesssim 1/n;$$ in particular, $C_1\le\diam W_n(F)\le C_2$ for some finite and positive constants $C_1$ and $C_2$.

Besides, by the separation hypothesis, let $\{{R}_1, \ldots, {R}_{m'}\}$ be an IFS for $F$ satisfying the SOSC with open set $U$.
We denote by $e_i$ the contraction ratio of $R_i$ and by $\mathcal{J}^\ast$ the corresponding set of finite words in the alphabet $\mathcal J=\{1, \ldots, m'\}$. Also, let $e_\ast=\min\{e_i:1\le i\le m'\}$. 
Recall that since the OSC holds, then there is a constant $L$ such that for any set $A\subset\rr^d$, the family 
\begin{equation*}
\mathcal J_A:=\{\alpha=(i_1,\ldots, i_k)\in\mathcal J^\ast: e_{\alpha}<\diam A\le e_{(i_1,\ldots,i_{k-1})}, A\cap R_\alpha(F)\neq\emptyset\}
\end{equation*} has at most $L$ elements; see \cite{BG}. 

For any $n$, define $A_n=S_{\beta_n}(W_n(F))$, so that $A_n\subset F$ by (\ref{subsets}). Note that for any $\alpha\in\mathcal J^\ast$ we have 
\begin{equation}\label{contencion}
A_n\cap R_\alpha(U)\subset R_\alpha(F),
\end{equation}
 which follows because the union in (\ref{OSC}) is disjoint and $F\subset \overline{U}$.
Below we show that there exists $\eta>0$ such that, for any $n$, we have 
\begin{equation}\label{c}
B(x_n,e_\alpha\eta)\subset  R_\alpha(U)
\end{equation} 
for some $x_n\in A_n$ and some $\alpha\in\mathcal J_{A_n}$. This allows to show later that the sequence of compact sets $\{R_\alpha^{-1}(A_n\cap B(x_n,e_\alpha\eta))\}_n$, which have diameters bounded away from zero and are contained in $F$, has a limit point that contains a Cartesian product satisfying the conclusion of the statement of the theorem.   

Firstly, the open set $U$ is chosen so that 
\begin{equation}\label{openful}
\nu(U)=\nu(F),
\end{equation} 
where $\nu$ denotes the restriction of the $s$-dimensional Hausdorff measure to $F$ and $s$ is the Hausdorff dimension of $F$; this is possible because of \cite[Theorem 2.2]{S94} and since the corresponding self-similar measure with natural weights is a multiple of $\nu$ by \cite[Section 5.3]{Hut}. Note in particular that $0<\nu(F)$. Now, given that the sets 
\begin{equation*}
\Omega_k:=\{x\in U: (k-1)^{-1}<d(x, U^c)\}
\end{equation*} 
increase to $U$, we fix $k_0$ such that 
\begin{equation*}
\nu(U\setminus\Omega_{k_0})\le\nu(F)/(2C_2L).
\end{equation*}
Also, for any $\alpha \in\mathcal J_{A_n}$ we have  
\begin{equation*}
C_1e_\ast c_{\beta_n}\le e_\alpha< C_2c_{\beta_n}.
\end{equation*} 
Then, by (\ref{subsets}),  the scaling property of the Hausdorff measure and since $\nu(\partial U)=0$ by (\ref{openful}) (where $\partial U$ is the topological boundary of $U$), we get
\begin{align*}
c_{\beta_n}\nu(F)=\nu(S_{\beta_n}(F))\le\nu(A_n)&=\sum_{\alpha\in\mathcal J_{A_n}}\nu(A_n\cap R_\alpha(F)) \\
&\le	\left(\sum_{\alpha\in\mathcal J_{A_n}}\nu(A_n\cap R_\alpha(\Omega_{k_0}))\right)+c_{\beta_n}\nu(F)/2.
\end{align*}
This implies
that there is $\alpha_n\in\mathcal J_{A_n}$ such that 
$A_n\cap R_{\alpha_n}(\Omega_{k_0})\neq\emptyset$,
and picking $x_n$ in  this intersection, we get $B(x_n,e_{\alpha_n}\eta)\subset R_{\alpha_n}(U)$ with $\eta=k_0^{-1}$. This proves (\ref{c}).

After choosing a subsequence and relabeling, we assume that $c_{\beta_n}^{-1}e_{\alpha_n}\to \tilde c$, for some positive and finite constant $\tilde c$, and that  $T_n(x_n)\to z$ for some $z\in W+F$. Also, define 
\begin{equation*}
B_n:=T_n(A_n\cap B(x_n,e_{\alpha_n}\eta))=W_n(F)\cap B(T_n(x_n),c_{\beta_n}^{-1}e_{\alpha_n}\eta)
\end{equation*} 
and let $P\subset W+F$ be any limit set in the Hausdorff metric of the sequence $\{B_n\}$. Given $\epsilon>0$, it follows that $(W+F)\cap B(z,\tilde c\eta/2)\subset [B_n]_\epsilon$ for all $n$ large enough, and consequently, $(W+F)\cap B(z,\tilde c\eta/2)\subset P$. Then, by Lemma \ref{lemma:prod}, there are sets $Q'\subset V_{\mathcal O}$ and $C'$ as in the statement of Theorem \ref{Theorem:product-structure} with $Q'\times C'\subset P$. 

Given that
\begin{align*}
d_{\mathcal H}\bigl(R_{\alpha_n}^{-1}\circ S_{\beta_n}(B_n),R_{\alpha_n}^{-1}\circ S_{\beta_n}(P)\bigr)=e_{\alpha_n}^{-1}c_{\beta_n}d_{\mathcal H}(B_n, P)\to 0
\end{align*}
and that $R_{\alpha_n}^{-1}\circ S_{\beta_n}(B_n)\subset F$ for all $n$ by (\ref{contencion}), we obtain that any limit point of $\{R_{\alpha_n}^{-1}\circ S_{\beta_n}(P)\}_n$, in the Hausdorff metric, is contained in $F$.

Let $O(d)$ be the group of $d\times d$ orthogonal matrices, and let $O_n\in O(d)$ and $b_n\in \rr^d$ be the orthogonal and translation parts of the similarity $R_{\alpha_n}^{-1}\circ S_{\beta_n}$. After picking a subsequence and relabeling, we assume by compactness of $O(d)$ that $O_n\to O$ for some $O\in O(d)$, and also that $b_n\to b$ for some $b\in\rr^d$, which is possible since $R_{\alpha_n}^{-1}\circ S_{\beta_n}(P) \Subset [F]_1$ for all $n$ sufficiently large. 
Defining the similarity $T(x)=\tilde{c} Ox +b$, then $R_{\alpha_n}^{-1}\circ S_{\beta_n}$ converges uniformly to $T$ on bounded sets, and in particular on $P$. Therefore, $T(P)$ is a limit point of $\{R_{\alpha_n}^{-1}\circ S_{\beta_n}(P)\}_n$ in the Hausdorff metric, and in consequence, $T(P)\subset F$. 

Applying  Corollary \ref{Coro:invariance}, note that $V_{\mathcal O}$ is invariant by $O_n$ for all $n$, and hence by $O$. Then, $V_{\mathcal O}^\perp$ is invariant by $O$. These observations, together with the fact that $T$ is a similarity, finally show that the set $T(P)$ contains a similar copy $Q\times C$ of $Q'\times C'$, with  $Q\subset V_{\mathcal O}$ a non trivial cube, $C\subset V_{\mathcal O}^\perp$, and $\dim_A C=\dim_A \pi_{V_{\mathcal O}^\perp }F$. This concludes the proof.
\end{proof}

\section*{Acknowledgments} I am grateful to Kathryn Hare for valuable discussions at an early stage of this project. I also thank the referee for a number of helpful comments.

\end{document}